\documentclass[11pt]{article}
\usepackage{amsmath,amsthm,amsfonts,amssymb,amscd}
\usepackage[latin1]{inputenc}

\usepackage{graphicx}

\usepackage{amsmath}

\usepackage{amssymb}

\usepackage{amscd}

\usepackage{bbm}

\newtheorem{Theorem}{Theorem}[section]
\newtheorem{Corollary}[Theorem]{Corollary}
\newtheorem{Proposition}[Theorem]{Proposition}
\newtheorem{Lemma}[Theorem]{Lemma}
\newtheorem{Claim}[Theorem]{Claim}
\theoremstyle{Definition}
\newtheorem{Definition}[Theorem]{Definition}
\newtheorem{Example}[Theorem]{Example}

\theoremstyle{Remark}
\newtheorem{Remark}[Theorem]{Remark}

\def\leaderfill{\leaders\hbox to .8em{\hss .\hss}\hfill}
\def\_#1{{\lower 0.7ex\hbox{}}_{#1}}

\def\C{{\mathcal{C}}}

\def\G{{\mathcal{G}}}
\def\Sa{{\mathcal{S}}}
\def\fa{{\mathcal{F}}}

\def\U{{\mathcal{U}}}

\def\po{{\partial}}

\def\vr{{\varphi}}

\def\la{{\lambda}}

\def\ov{\overline}

\def\ve{{\varepsilon}}

\def\re{{\mathbb{R}}}

\def\Cent{\operatorname{{Cent}}}
\def\Sad{\operatorname{{Sad}}}

\def\Hol{\operatorname{{Hol}}}

\def\Diff{\operatorname{{Diff}}}
\def\sing{\operatorname{{sing}}}
\def\Sing{\operatorname{{sing}}}
\def\codim{\operatorname{{codim}}}

\title{\LARGE {Codimension one foliations with Bott-Morse singularities II}}

\author{{Bruno Sc\'ardua and Jos\'e Seade} }
\date{}

\begin{document}

\maketitle

\begin{abstract}
 We study codimension one  foliations with singularities
  defined locally by Bott-Morse functions   on  closed oriented manifolds\footnote[0]{Subject Classification: Primary
57R30, 58E05; Secondary 57R70, 57R45.
\newline Keywords: Bott-Morse foliations, holonomy, basin, stability, surgery. \newline
Partially supported by CNPq, Brazil, CONACYT and DGAPA-UNAM,
Mexico.}.
 We carry to this setting the classical concepts of
holonomy of invariant sets and stability,  and prove  a stability
theorem in the spirit of the  local  stability theorem of Reeb.
This yields, among other things, a
 good topological understanding  of the leaves one may
have around a center-type component of the singular set, and also
of  the topology of its {\it basin}. The stability theorem further
allows the description of
 the topology of the boundary of the basin   and how the topology of the leaves changes when
passing from inside to outside the basin. This is described via
{\it fiberwise Milnor-Wallace surgery}. A key-point for this is to
show that if the boundary of the basin of a center is non-empty,
then it contains a saddle;
 in this case we say that the center and the saddle are {\it paired}.
 We then describe the possible pairings one may have in dimension three
 and use a construction  motivated by the classical saddle-node bifurcation,
 that we call {\it foliated surgery}, that allows  the
 reduction of certain pairings of singularities of a foliation. This is used together with our
 previous work on the topic to prove
 an extension for 3-manifolds of Reeb's sphere recognition theorem.
\end{abstract}

\tableofcontents

\section{Introduction}
\label{section:Introduction}

Foliations with singularities on smooth manifolds appear naturally
in many fields of mathematics. For instance, given a manifold $M$
one may: i)  consider a smooth function $f: M \to \mathbb R$ and
look at its level sets; the critical points of $f$ yield
singularities of the corresponding foliation; ii) consider a Lie
group action $G \times M \to M$; the $G$-orbits define a foliation
with singularities at the points in special orbits; iii) if $M$
admits a Poisson structure, then this structure determines a
foliation by symplectic leaves, which is singular at the points
where the rank drops.

These are just a few examples of singular foliations, there are
many more.  And yet, in the setting we envisage here -that of real
manifolds- our knowledge of singular foliations is not so big,
except in special cases.  No doubt, a reason for this is the high
degree of difficulty involved in the study of singular foliations:
if non-singular foliations theory is already hard enough, adding
singularities can turn it beyond any reasonable scope.

Thus one is naturally lead to imposing certain restrictions on the
type of singular foliations one studies. Here we look at a class
of codimension 1 foliations which is a natural generalization of
the Morse foliations. We look at foliations which are locally
defined by Bott-Morse functions. This class is large enough  to be
a rich family, including many interesting examples,
 and yet the conditions we impose do allow  certain control, and one may  hope to say interesting things.

The concept   of  foliations with Bott-Morse
singularities on smooth manifolds  obviously comes  from
 the landmark work of M. Morse in his Colloquium Publication \cite{Morse1}, as well as
from R. Bott's
generalization of Morse functions in \cite{Bott1}.
This notion for foliations  was introduced in our previous
article \cite {Se-Sc}, where we focused on the case where all
singularities were transversally of center-type. The presence of
saddles obviously makes  the theory richer, and this is the
situation we now envisage.

This type of foliations fits  within the framework of {\it
generalized foliations} introduced by H. J. Sussmann in
\cite{Sus}, and they appear naturally in various contexts. For
instance every singular Riemannian foliation in the sense of P.
Molino \cite {Mo} has Bott-Morse singularities of center-type;
this includes the foliations given by cohomogeneity 1 actions of compact
Lie groups. The foliations given by generic singular Poisson
structures of corank 1 also have Bott-Morse singularities, and so
do the liftings to fiber bundles of Morse foliations given on the
base.

The philosophy underlying this article is the following. Suppose
$\fa$ is a codimension 1, transversally oriented closed foliation
on a closed, oriented and connected $m$-manifold $M$, and the
singularities of $\fa$ are all of Bott-Morse type. One has that
the transverse type of the foliation at each component $N_j$ of
its singular set is independent of the choice of transversal, and
it is therefore determined by its Morse index. Centers correspond
to the extreme cases of Morse index $0$ or $m -\hbox{dim}\,N_j$.
We know from our previous article that if there are only
center-type components, then there are exactly two such
components, and the foliation is given by the fibers of a smooth
Bott-Morse function $f:M \to [0,1]$ with  two critical values at
the points $\{0,1\}$, which correspond to the two components of
the singular set. Furthermore, in this situation the leaves are
all spherical fiber bundles over each component of the singular
set. Hence the topology is ``well understood" in these cases.

One may thus think of the general case in the following way.
Assume there is a center-type component $N_0\subset \sing(\fa)$
and look at the leaves around it, which we can describe as above;
this uses a local stability theorem similar to Reeb's theorem for
compact leaves (Section \ref{Section:Basins}). Look at the {\it
basin} $\mathcal C(N_0,\fa)$ of $N_0$ (see the definition inside);
these are the leaves that ``we understand". Now we go to the
boundary $\partial \mathcal C(N_0,\fa)$ of this set. By
construction, if  $\partial \mathcal C(N_0,\fa)= \emptyset$ then
$M=\mathcal C(N_0,\fa)$,  there are no saddle  singularities of
$\fa$ in $M$ and we are in the situation previously envisaged in
\cite{Se-Sc}. Otherwise there must be a saddle component $N_1$ in
$\partial \mathcal C(N_0,\fa)$ and no center-components there; in
this case we say that $N_0$ and $N_1$ {\it are paired}, a
key-concept for this article, inspired by \cite{Camacho-Scardua}.

If the Morse index of $N_1$ is not $1$ nor $m - \hbox{dim}\,N_1
-1$, then   $N_1$ has exactly one separatrix $L_1$ and if we
assume that there are no saddle-connections, then the compact set
$\Lambda(N_1) = N_1 \cup L_1$ is precisely the boundary $\partial
\mathcal C(N_0,\fa)$. In sections \ref{Section:distinguishedneigh}
and \ref{sec. Partial Stability Theorem}  we prove a Stability
Theorem (\ref{Theorem:partialstabilitytheorem}) that allows us to
determine the topology of $\partial \mathcal C(N_0,\fa)$ and of
the leaves in a neighborhood of $\partial \mathcal C(N_0,\fa)$,
but outside the basin, by comparison with the leaves inside the
basin.
  This is possible thanks to the triviality
of the holonomy of codimension one invariant subsets that we prove
in Section \ref{Section:holonomy-singular-set}, and the
description of {\it distinguished  neighborhoods} of saddles in
Section \ref{Section:distinguishedneigh}. The description we give of   the topology of $\partial \mathcal
C(N_0,\fa)$ and the topology of the leaves beyond the boundary
  is done by means of    ``fiberwise Milnor-Wallace surgery" (Theorem
\ref{topology of boundary}). Of course this is  inspired by ideas
of R. Thom in \cite {Thom}.

When the Morse index of $N_1$ is $1$ or $m - \hbox{dim}\,N_1 -1$
 then the saddle $N_1$ may have one or two separatrices. If it has only one
 separatrix, then the discussion is exactly as above. However, if $N_1$ has
 two distinct separatrices $L_1, L_2$, then $\Lambda(N_1) = N_1 \cup L_1 \cup L_2$
 has two components that meet at $N_1$ and the topology of the leaves in
$C(N_0,\fa)$ only determines the topology of one of these
components. We need more information in order to  determine the
topology of all of $\Lambda(N_1)$. As we shall see, this is
actually equivalent to determining the topology of the leaves ``in
the other side" of the invariant set $\Lambda(N_1)$.

The hope would be that there is not ``much more" beyond the
boundaries of the basins of  center-type components. However
examples show that this is not the case and the possibilities are
infinite. Therefore we must impose additional restrictions on the
foliations we consider in order to be capable of saying something.
The first natural condition, that we already used above, is
 that the foliation has no saddle connections; we call
these {\it Bott-Morse foliations}, see definition
\ref{Definition:Bott-Morse}. We also restrict this discussion to
closed foliations, which already leaves out one the main features
of foliation-theory: the presence of recurrences. Yet, for
instance, if a 3-manifold admits a Heegard splitting of genus $g >
0$, then it admits closed Bott-Morse foliations with $2g$
non-isolated center components, $2g-2$ isolated saddles and leaves
of all genera $1,...,g$ (see Section \ref{Section:examples}).
Moreover, if $M^3$ has a Heegard splitting of genus $g$ then it
has Heegard splittings of all genera $\ge g$, and therefore
Bott-Morse foliations with leaves of all genera. In higher
dimensions the situation is even wilder.

We thus have that the ``zoo" of Bott-Morse foliations is rather
big. So in the last sections of this work we restrict ourselves to
dimension three, which is already rich enough.  We describe
several types of possible pairings one may have in 3-manifolds,
completing the study done in \cite{Camacho-Scardua} for foliations
with isolated Morse singularities, and we show that if one imposes
the condition of having ``sufficiently more" centers than saddles,
 then the above mentioned list   actually   describes all possible pairings.

We then extend to Bott-Morse foliations a reduction technique
introduced in \cite{Camacho-Scardua}, that we call {\it foliated
surgery}, that allows us to reduce, under appropriate conditions,
the number of components of the singular set of a foliation.  This
technique   is used in \cite{Camacho-Scardua, Camacho-Scardua2} to
prove that a closed $m$-manifold supporting a {\it Morse
foliation} with strictly more centers than saddles is homeomorphic
to the $m$-sphere $S^m$  or to a {\it Kuiper-Eells manifold}, {\it
i.e.}, a closed manifold supporting a Morse function with three
singular points.  This generalizes the classical ``sphere
recognition theorem" of Reeb and an analogous result by
Kuiper-Eells.

In this article we follow the same strategy, which is now
harder because the singular set of the foliation may have
dimension $> 0$ and components of different dimensions. The idea
is the following. Let $N\subset \sing(\fa)$ be a center type
component. If $\fa$ is not compact then  $\partial \mathcal
C(N,\fa)$, the boundary of its basin, contains some saddle
singularity $N_0\subset \sing(\fa)$. Denote by $\Lambda(N_0)$ the
union of $N_0$ and all separatrices of $N_0$. Since $m$ is $3$,
there are cases  where $\Lambda(N_0) \setminus N_0$ is not
connected. Then the ``external leaves" $L_e$, those ``beyond" the
boundary $\partial \mathcal C(N,\fa)$,  are
  such that  $L_e \setminus (L_e\cap W)$ has
 two connected components, where $W$ is a distinguished neighborhood of $N$. In this case
our Stability Theorem~\ref{Theorem:partialstabilitytheorem} only
gives information about  the connected component of $L_e\setminus
(L_e \cap W)$  which is close to the basin $\C(N, \fa)$. The
topology of the other connected component has to be controlled by
some other additional information. Thence we must demand having
``sufficiently more" centers than saddles:  $c(\fa) > 2 s(\fa)$.
With this hypothesis, the local description of $\fa$ in $W$ that
we get allows us to describe the topology of $L_e\cap W$. We can
then use Theorem \ref{topology of boundary}   to describe the
topology of the external leaves $L_e$.

This description allows the
classification of the pairings that may possibly appear in this
process, and we  use foliated surgery (cf.
\S~\ref{Section:foliatedsurgery}) in order to reduce the total
number of singularities of $\fa$ but preserving the inequality
$c(\fa)> 2s(\fa)$. Repetition of this process finally shows that
$M$ can be equipped with a {\it compact} Bott-Morse foliation, so
we use \cite{Se-Sc} to conclude  (Theorem
~\ref{Theorem:Center-Saddledimensionthree}) that if $M^3$ is a
closed oriented connected $3$-manifold that admits a closed
 Bott-Morse foliation $\fa$ satisfying $c(\fa) > 2 s(\fa)$, then
 $M$   is the 3-sphere, a product  $S^1 \times S^2$ or
  a Lens space $L(p,q)$.

When the singularities are all of the same dimension, then the
hypothesis $c(\fa) >  s(\fa)$ is enough to get all the information
we need in order  to carry on with the same strategy, and we
arrive to the same conclusion (Theorem
\ref{Theorem:Center-Saddledimensionthreecircles}).

Simple  examples (in Section~\ref{Section:examples}) show that the
above bounds   ($c(\fa) > 2 s(\fa)$ in general or just $c(\fa) >
s(\fa)$ if the singular set is ``pure dimensional") are
  sharp in the sense that there are many closed, oriented
  3-manifolds that admit closed Bott-Morse foliations with
   $c(\fa) = 2 s(\fa)$ singular components of mixed dimensions,
   and also with $c(\fa) = s(\fa)$ singular components of the same dimension (either $0$ or $1$, as we please).

\vglue.1in This work was done during  visits of the first named
author to the Instituto de Matem\'aticas of UNAM in Cuernavaca,
Mexico, and  visits of the second named author to IMPA, Rio de
Janeiro, Brazil. The authors want to thank these institutions for
their support and hospitality.

%%%%SECTION 1

\section{Definitions and examples}

Throughout this article $M$ is a closed, oriented,  smooth manifold  (say of
class $C^\infty$ for simplicity), endowed with a riemannian metric. As usual,
being closed means that $M$ is compact and has empty boundary. Let
us  recall  the following definition (cf. \cite{Morse1, Bott1}).

\begin{Definition} {\rm
A smooth function $f: M \to \mathbb R$ is {\it Bott-Morse}  if its
critical points form a union of disjoint, closed, submanifolds $
\bigcup\limits_{j=1}^t N_j$ of $M$ which are non-degenerate for
$f$, {\it i.e.}, for each $p \in N_j \subset \sing(\fa)$ and for
each small  disc $\Sigma_p $ transversal  to some  $N_j$ of
complementary dimension, one has that the restriction
$f|_{\Sigma_p}$ has an ordinary Morse singularity.}
\end{Definition}

 Now we have:

\label{Section:Definition} Let $\fa$ be a codimension one smooth
foliation with singularities on a manifold $M$ of dimension $m \ge
2$. We denote by $\sing(\fa)$ the singular set of $\fa$.

\begin{Definition}[Bott-Morse singularity]
\label{Definition:bott-morsesing}{\rm   The singularities of $\fa$
are of {\it Bott-Morse type\/} if $\sing(\fa)$ is a disjoint union
of a finite number of closed    connected submanifolds, $\sing(\fa)
= \bigcup\limits_{j=1}^t N_j$, each of codimension $\ge 2$,
 and for each $p \in N_j \subset \sing(\fa)$ there exists a neighborhood $V$ of
$p$ in $M$ where $\fa$ is defined by a Bott-Morse function.}  \end{Definition}

\medskip

 That is, if  $n_j$ is the dimension of $N_j$, then there
is a diffeomorphism $\vr\colon V \to P \times D$, where $P \subset
\re^{n_j}$ and $D \subset \re^{m-n_j}$ are discs centered at the
origin, such that $\vr$ takes $\fa|_V$ into the product foliation
$P \times \G$, where $\G=\G(N_j)$ is the foliation on $D$ given by
some function with a  Morse singularity at the origin.

In other words,  $\sing(\fa) \cap V = N_j \cap V$; \, $\vr(N_j
\cap V) = P\times\{0\} \subset P \times D \subset \re^{n_j} \times
\re^{m-n_j}$, and we can find coordinates $(x,y) =
(x_1,\dots,x_{n_j,} y_1,\dots,y_{m-n_j}) \in V$ such that $N_j
\cap V$ is given by $\big\{y_1=\cdots=y_{m-n_j}=0\big\}$ and
$\fa|_V$ is given by the levels of a function $J_{N_j}(x,y) =
\sum\limits_{j=1}^{m-n_j} \la_j\,y_j^2$ where $\la_j \in \{\pm
1\}$.

\begin{figure}[ht]
\begin{center}
\includegraphics[scale=0.4]{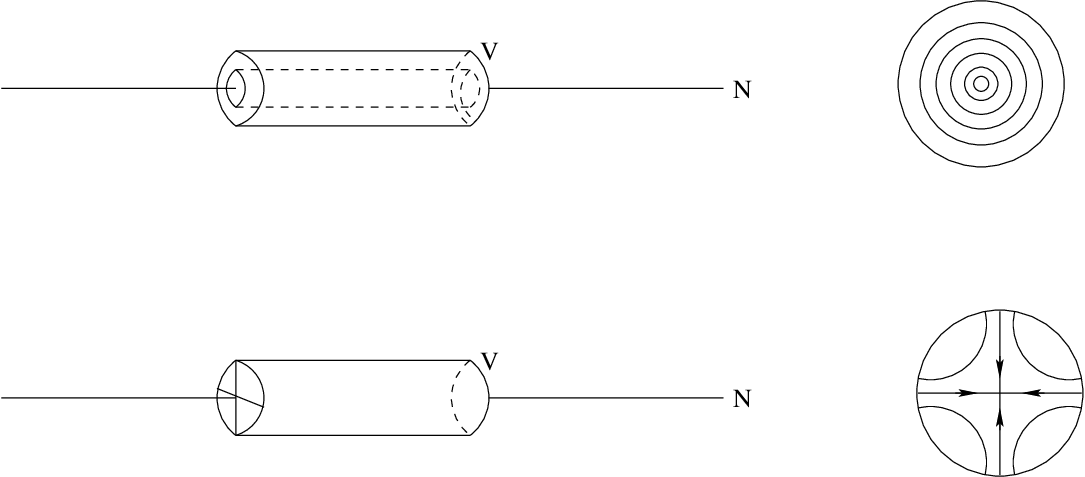}
\end{center}
\caption{Center and saddle type Bott-Morse singularities}
\end{figure}

%%%\centerline{\bf Figure 1} \vglue.1in

 The discs $\Sigma_p = \vr^{-1}(x(p)\times D)$ are
transverse to $N_j$ and they are transverse to $\fa$ outside
$N_j$. The restriction $\fa|_{\Sigma_p}$ is an ordinary Morse
singularity and its Morse index does not depend on the point $p$
in the component $N_j$,
 nor on the choice of the transversal slice  $\Sigma_p$.
 We  refer to $\G(N_j) = \fa|_{\Sigma_p}$ as the {\it
transverse type\/} of $\fa$ along $N_j$, and its {\it Morse index} is
 the  Morse index of $f$ at $p$.

\begin{Definition}
\label{Definition:centersaddlecomponent} {\rm A component $N_j
\subset \sing(\fa)$ is of {\it center-type} (or just {\it a
center}) if the transverse type $\G(N_j) = \fa|_{\Sigma_p}$ of
$\fa$ along $N_j$ is a center, {\it i.e.}, its  Morse index
$r=r_j$
 is $0$ or $m - n_j$.
The component $N_j$ is  of  {\it saddle-type} (or just {\it  a
  saddle}) if
its transverse type is a saddle, {\it i.e.}, its Morse index $r$
is $\ne 0, m - n_j$.}
\end{Definition}

 In a neighborhood of a center the leaves of $\fa$ in the transversal $\Sigma_p$ are
diffeomorphic to $(m- n_j-1)$-spheres, where $n_j$ is the
dimension of $N_j$. In a neighborhood of a saddle
 we have  leaves called {\it separatrices\/} of $\fa$,
which on the transversal disc are conical leaves given by
expressions of the form $y_1^2 +\cdots+ y_r^2 = y_{r+1}^2 +\cdots+
y_{m_j}^2\ne0$; each such leaf contains $p$ in its closure. Given
a component $N_j \subset \sing(\fa)$
 whose transverse type is
a saddle, {\it a separatrix of $\fa$ through $N$\/} (or simply a
{\it separatrix of $N$})  is a leaf $L$ such that its closure $\ov
L$ contains $N_j$.
 This means that $L$ meets each small
$(m-n_j)$-disc $\Sigma$ transversal to $N_j$ in a separatrix of
$\fa|_\Sigma\,$.

As in the case of isolated singularities, these concepts do not
depend on the choice of orientations. We denote by
$\Cent(\fa)\subset \sing(\fa)$ the union of center-type components
in $\sing(\fa)$, and by $\Sad(\fa)$ the corresponding union of
saddle components.  We denote by $c(\fa)$ and $s(\fa)$ the number
of connected components in $\Cent(\fa)$ and $\Sad(\fa)$
respectively.

We say that $\fa$ is {\it compact} if every leaf of $\fa$ is
compact (and consequently $s(\fa)=0$). The foliation $\fa$ is {\it
closed} if every leaf of $\fa$ is closed off $\sing(\fa)$. In this
case, if $M$ is compact, then all leaves are compact except for
those containing separatrices of saddles in $\Sad(\fa)$: such a
leaf is contained in the compact singular variety $\ov{L}  \subset
L \cup \Sad(\fa)$, union of $\ov{L}$ and the saddle components for
which it is a separatrix. A closed foliation on a compact manifold
is compact if and only if $s(\fa)=0$ (see \cite {Se-Sc} for
details).

%Let $N\subset \Cent(\fa)$ be a component of dimension $k$. Suppose
%that the nearby leaves of $\fa$ are compact. We define
%$\Omega(N,\fa)=\Omega(N)\subset M$ as the union of $N$ and all the
%leaves $L \in \fa$ which are compact and bound a compact invariant
%region $R(L,N)$ which is a neighborhood of $N$ in $M$. The region
%$R(L,N)$ is equivalent to a fibre bundle with fibre the closed
%disc $\ov D^{m-k}$ over $N$. The foliation
%$\fa\big|_{R(L,N)\setminus N}$ is transverse to the above
%fibration.

 We say
that $\fa$ {\it has a saddle-connection\/} if there are components
$N_1 \ne N_2$ of $\Sad(\fa)$ and a leaf $L$ of $\fa$ which is
simultaneously a separatrix of $N_1$ and $N_2$\,. If a leaf $L$ is
a separatrix of $\fa$ through $N$ and $L$ meets some transversal
$(m-n)$-disc  $\Sigma$ in two distinct separatrices of
$\fa|_\Sigma\,$,
 then we say $L$ is a {\it
self-saddle-connection\/} of $\fa$.

\begin{Definition} [Transverse orientability]
\label{Definition:orientability} {\rm Let $\fa$ be a $C^\infty$
codimension one foliation with Bott-Morse singularities on $M^m$,
$m \ge 2$.  We  say that $\fa$ {\it is transversally orientable}
if  there exists a $C^\infty$ vector field $X$ on $M$, possibly
with singularities at $\sing(\fa)$, such that $X$ is transverse to
$\fa$ outside $\sing(\fa)$.}
\end{Definition}

\begin{Definition} [Bott-Morse foliation]
\label{Definition:Bott-Morse} {\rm Let $\fa$ be a $C^\infty$
codimension one foliation on a differentiable manifold $M^m$, $m
\ge 2$. We       say that $\fa$ is a {\it Bott-Morse foliation}
if:
\begin{itemize}
\item[(i)] The singularities of $\fa$ are of Bott-Morse type.
\item[(ii)] $\fa$ has no saddle-connections on $M$ ($\fa$ may have
self-saddle-connections). \item[(iii)] $\fa$  is transversally
orientable.
\end{itemize}}
\end{Definition}

In this paper, we will be mostly dealing with {\bf closed} Bott-Morse foliations.

\subsection{Examples}
\label{Section:examples}

\begin{Example} {\rm If $\pi: M^{n+k} \to B^n$ is a $C^\infty$
fibered bundle and the manifold $B$ is equipped with a Bott-Morse
foliation, then the inverse image of the leaves determines a
Bott-Morse foliation on $M$. In particular, if $f: B \to \re$ is a
Morse function with no saddle connections, then its level surfaces
determine a Morse foliation on $B$ which lifts to a Bott-Morse
foliation on $M$. }
\end{Example}

\begin{Example}
{\rm  A {\it Poisson structure} on a smooth manifold $M$ consists
of  vector bundle morphism $\psi \colon T^*M\to TM$ satisfying an
integrability condition, whose rank at each point is called the
rank of the Poisson structure. If the rank is  constant then the
integrability condition implies one has a foliation on $M$, of
dimension equal to the rank;  the tangent space of the foliation
at each point $x \in M$ is  the image of $\psi(T^*_xM)$ in $T_xM$.
If the
 rank is not constant then one has a singular  foliation
 with singularities at the points where the
rank drops. The Dolbeault-Weinstein theorem implies that at
such points the transversal structure plays a key-role,
and generically the transverse structure is given by a Morse function.
}
\end{Example}
\begin{Example}  {\rm
Every  codimension 1 singular Riemannian foliation in the sense of
P. Molino (see for instance the last chapter of his book \cite
{Mo}; see also \cite {Hur-Tob} for more on the subject), is
Bott-Morse with only center-type components. This includes the
foliations defined by cohomogeneity 1 isometric actions of Lie
groups on smooth manifolds with special orbits.}
\end{Example}

Let us now give some  examples of Bott-Morse foliations on
3-manifolds which are important in the sequel.

\begin{Example}\label{multiple irreducible}
 {\rm Every closed oriented $3$-manifold can be
expressed as a union $M^3 = L(g) \cup L(g)'$ where $L(g), L(g)'$
are solid handlebodies of genus $g \geq 0$, glued along their
boundary $S_g$. These are called {\it Heegard splittings {\rm(or}
decompositions{\rm)}} of $M$, and $g$ is the genus of the corresponding
decomposition. The sphere is the only $3$-manifold  admitting such
a splitting with genus $0$. If $M$ has a splitting of genus $g$ then it
has splittings of all genera $\ge g$.

Given a Heegard  splitting $M^3 = L(g) \cup L(g)'$ one can take a
product neighborhood $\Sa_g \times [-\epsilon,\epsilon]$ of
$\Sa_g$ and foliate it by surfaces of genus $g$, parallel to the
boundary, with $S_g$ corresponding to $S_g \times \{0\}$. On the
level $\Sa_g \times \{\epsilon \}$ take circles
 $C_1,...,C_{g-1}$ separating this surface into $g$ components,
each of genus $1$, and deform each of these circles  to a point (see Figure \ref{Figure:example2.8.eps}).
We get inside $L(g)$ a singular surface $S$ with $g-1$
saddle-points, which splits $L(g)$ into $(g+1)$-components: an
``outer" one, diffeomorphic to $\partial S_g \times [0,\epsilon)$;
and $g$ ``inner" components, each diffeomorphic to
 an open solid torus $T_j, \, j=1,...,g$. We can
now foliate each $T_j$ in the usual way, by copies of $S^1 \times
S^1$, having in each a circle $N_j$ as singular set, all of
center-type. We can do the same construction on the other
handle-body $L(g)'$ and get a foliation on $M^3$ with Bott-Morse
singularities.

Notice this foliation has the surface $S$ as a
separatrix through each singular point, so $\fa$  has
saddle-connections, but it is easy to change the construction
slightly to get a Bott-Morse foliation. For instance, in the above
construction, deform only the first circle $C_1$ to a point,
getting a surface $S_1$ with one saddle singularity and bounding
two "inner" components, one, say $T_1$, of genus $1$ and another
$L(g-1)$ of genus $g-1$. Foliate the ``exterior'' of $S_1$ as
before, by surfaces of genus $g$, and foliate the torus $T_1$ as
before, with a center-type singular set. Foliate also a
neighborhood in $L(g-1)$ of its boundary by surfaces of genus
$g-1$. Now choose one of these surfaces of genus $g-1$; choose
on it a circle that separates a handle from the others and repeat
the previous construction. One gets a new separatrix, a new torus
foliated by concentric tori, and an open solid region of genus
$g-2$, etc. We get finally a foliation on $ L(g)$ with $g-1$
separatrices, each with an isolated saddle singularity, $g$
foliated tori, each with a non-isolated center-type singularity,
and leaves of all genera $g-1$, $g-2$,..., $1$ filling out $L(g)$;
and similarly for $L(g)'$.
\begin{figure}[ht]
\begin{center}
\includegraphics[scale=0.50]{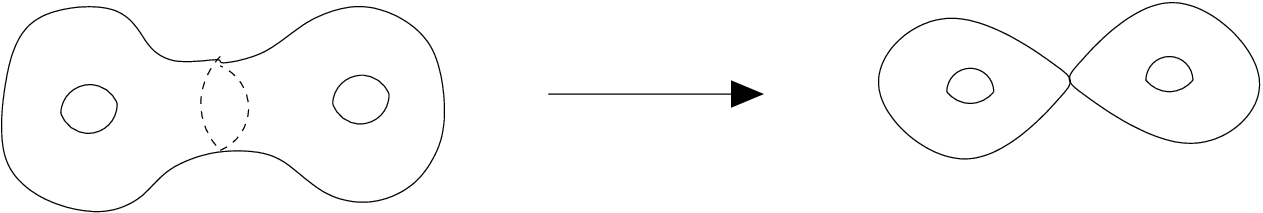}
\end{center}
 \caption{The construction of Example~\ref{multiple irreducible} for genus $g=2$. }
 \label{Figure:example2.8.eps}
\end{figure}

 }
\end{Example}

\begin{Example}
\label{Example:sphere} {\rm The sphere $S^3$ admits a Bott-Morse
foliation with singular set consisting of four isolated centers
and a non-isolated saddle as in  Figure~\ref{Figure:threesphere2};
the fourth center is at infinity. We can also foliate $S^3$ with
Bott-Morse singularities consisting of a non-isolated saddle, two
isolated centers and a non-isolated center
(Figure~\ref{Figure:threesphere1}). In both constructions an
isolated center is {\it at infinity} with respect to the saddle.

\begin{figure}[ht]
\begin{center}
\includegraphics[scale=0.35]{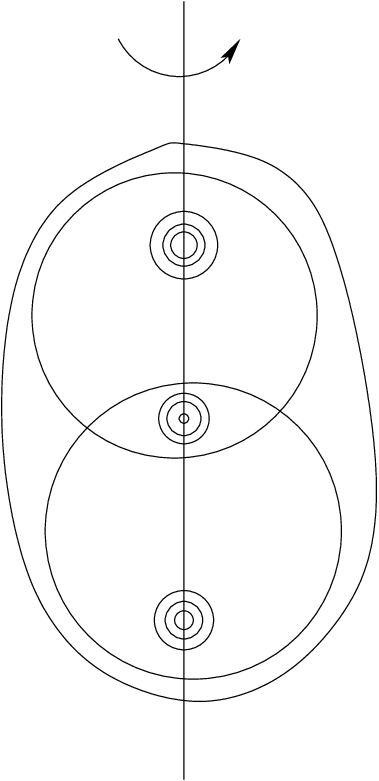}
\end{center}
 \caption{A foliation of $S^3$ by four isolated centers and a non-isolated saddle. }
 \label{Figure:threesphere2}
\end{figure}

\begin{figure}[ht]
\begin{center}
\includegraphics[scale=0.50]{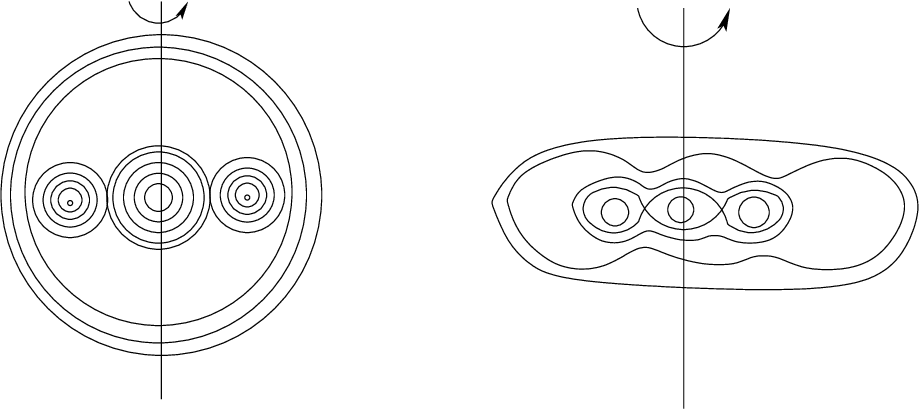}
\end{center}
\caption{A foliation of $S^3$ by two isolated centers, a
non-isolated saddle and a non-isolated
center.}\label{Figure:threesphere1}
\end{figure}

} \end{Example}

\begin{Example}
\label{Example:trivialfoliation} {\rm A closed 3-manifold $M$
admits a transversally oriented non-singular compact foliation
$\fa_0$ if and only if it fibers over the circle $S^1$. Every such
manifold can be equipped with a closed Bott-Morse foliation $\fa$
such that $c(\fa)=2\, s(\fa)=2\, k$, where $k\in \mathbb N$ is any
given natural number. In fact, given any codimension 1 foliation
on a 3-manifold $M$, one can replace a flow-box where the
foliation is regular, by a box where the new foliation has  two
isolated centers and one non-isolated saddle; this is depicted, in
reverse order, in Figure~\ref{Figure:notproduct1}. This process
can be repeated as many times as we want. If the original
foliation was closed, so is the new one.  Thus the equality
$c(\fa)=2\,s(\fa)$ is not restrictive (cf. Theorem
\ref{Theorem:Center-Saddledimensionthree} ): every closed,
oriented 3-manifold $M$ admits Bott-Morse foliations  with
$c(\fa)=2\,s(\fa) = 2k$ for each integer $k \ge 1$; and if $M$
fibers over $S^1$ then the foliation can be chosen to be closed.

Notice that in the previous construction the saddles and the centers
have
different dimensions.

\vglue.1in Similarly, given a closed manifold $M$ of dimension $m
\geq 2$, equipped with a non-singular compact foliation, we can
modify the foliation as follows. Choose a flow box region
$R\subset M$ and replace the foliation  in $R$ by a foliation with
an isolated center and an isolated saddle of type $x_m^2 -
\sum\limits_{j=1}^{m-1}x_j ^2=0$,   obtained by rotating the
Figure~\ref{Figure:centroseladim2trivial0} with respect to an axe
that passes through the center and the saddle.

\begin{figure}[ht]\label{fig:center-saddle}
\begin{center}
\includegraphics[scale=0.4]{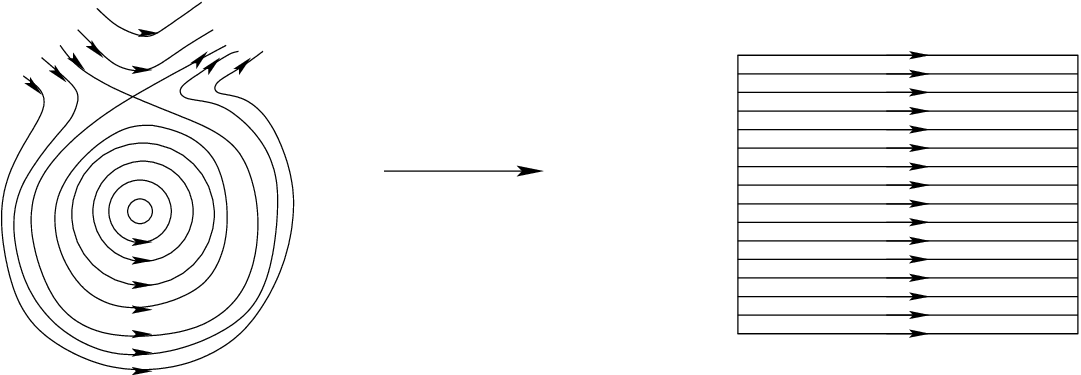}
\end{center}
\caption{} \label{Figure:centroseladim2trivial0}
\end{figure}
}\end{Example}

\noindent This produces a closed Bott-Morse foliation with $k$
isolated centers and $k$ isolated saddles. Now, we consider a
3-manifold $M$ equipped with a non-singular  compact
 foliation $\fa$. Let $L_0$ be a leaf of $\fa$ and $\gamma_0\subset
 L_0$ a $C^\infty$ closed path. The holonomy of $\gamma_0$ is
 trivial and we can take a  section $\Sigma$ transverse to
 $\gamma_0$, diffeomorphic to the square $Q=(-1,1)\times (-1,1)$,
 with $\Sigma\cap \gamma_0=\{p_0\}$, and a neighborhood $U$ of
 $\gamma_0$ diffeomorphic to the solid torus $S^1 \times Q$. In this
 neighborhood we can replace $\fa$ by the foliation obtained as
 the product  by $S^1$ of the two-dimensional foliation on the left in
 Figure~\ref{Figure:centroseladim2trivial0}. Repetition of this  process
 yields closed Bott-Morse foliations on $M$  with $k$
non-isolated centers and $k$ non-isolated saddles as singular set
(cf. Theorem \ref{Theorem:Center-Saddledimensionthreecircles}).

\begin{Example}
{\rm On the other hand, foliations with Bott-Morse singularities
satisfying the inequality $c(\fa) > 2 s(\fa)$ exist  in any 3-manifold that can be obtained by gluing two solid tori along their
boundary. These are the manifolds that
 appear in
Theorem~\ref{Theorem:Center-Saddledimensionthree}. For instance, in  these
cases one can foliate each torus by concentric tori and get a
foliation with two center components and no saddles, so $c(\fa) = 2$ and $ s(\fa) = 0$.

\begin{figure}[ht]
\begin{center}
\includegraphics[scale=0.45]{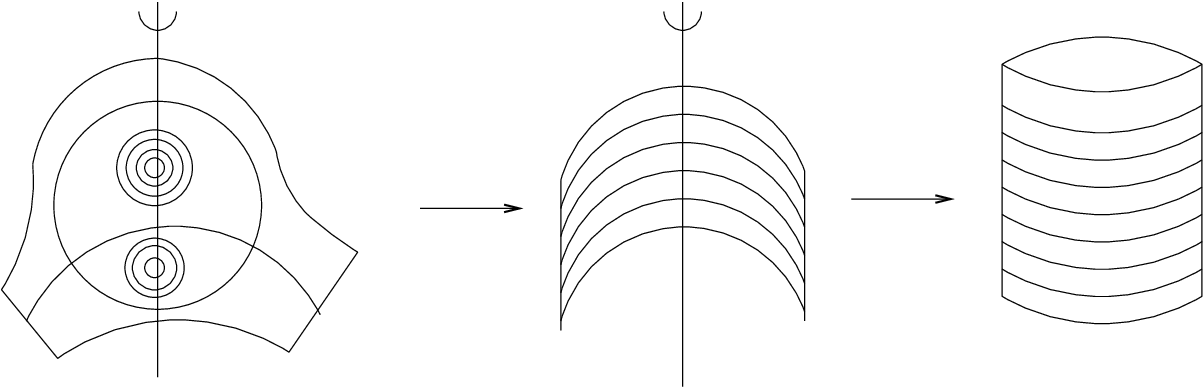}
\caption{{Reversing the arrows we get two isolated centers and a
non-isolated saddle from a flow box.}} \label{Figure:notproduct1}
\end{center}
\end{figure} }
\end{Example}

\begin{Example}
{\rm Consider now a closed $3$-manifold $M$ with a {\it
nonsingular} codimension one foliation $\fa_o$ with a Reeb
component $R\subset M$. In the solid torus region $R$ we replace
$\fa_o$ by a singular compact Bott-Morse foliation by concentric
tori. The resulting foliation $\fa$ on $M$ can be assumed smooth.
It is a Bott-Morse foliation with only center-type components but
it is not closed, and the manifold $M$ may not be as in
Theorem~\ref{Theorem:Center-Saddledimensionthree}. }
\end{Example}

\begin{Example} {\rm Let $T^2$ be a closed surface equipped with a
foliation $\fa_1$ defined by a Morse function with exactly three
singularities, two of center-type and one of saddle-type, and let
$M^3$ be an $S^1$-bundle over $T^2$. Then, as in the first example
above, $\fa_1$ lifts to a Bott-Morse foliation $\fa$ on  $M$ with
singular set consisting of three embedded circles, two of them
being (non-isolated) centers and the other a saddle. This
foliation satisfies the inequality on
Theorem~\ref{Theorem:Center-Saddledimensionthreecircles} and $M$
need not be as claimed in that theorem. Nevertheless, according to
Eells and Kuiper (\cite{Ee-Ku1}, \cite{Ee-Ku2}) the only
possibility for $T^2$ is $T^2=\mathbb RP(2)$ and the resulting
manifold $M$ is non-orientable.}
\end{Example}

%%%%%%%%%SECTION

\section{Holonomy of invariant subsets}

In this section we extend to Bott-Morse foliations the classical
notion of holonomy of leaves, that we recall below.

\subsection{Holonomy of a leaf}
This notion  is originally found in the work of Ehresmann and Shih
\cite{Er-Sh} and was further developed in the subsequent work of
Reeb \cite{Reebthesis}. Let $\fa$ be a codimension $k$ foliation
on a manifold $M$ of dimension $m=k+l$. A {\it distinguished}
neighborhood for $\fa$ in $M$ is an open subset $U\subset M$ with
a coordinate chart $\vr \colon U \to \vr(U)\subset \mathbb R^m$
such that $\vr(U)=D_1^l\times D_2^k$ is the product of discs in
$\mathbb R^m=\mathbb R^l\times \mathbb R^k$ and the  leaves of the
restriction $\fa\big|_U$ ({\it i.e.}, the plaques of $\fa$ in $U$)
are of the form $\vr^{-1}(D_1 ^l\times\{y\}), \, y \in D_2 ^k$. If
$V\subset U$ is another distinguished open set,  we say that $V$ is
{\it uniform} in $U$ if every plaque of $\fa$ in $U$ meets at most
one plaque of $\fa$ in $V$. This means that the natural map on
leaf spaces $V/\fa\big|_V \to U/\fa\big|_U$ is injective. In
codimension one every distinguished open set $V$ contained in
another distinguished open set $U$ is always uniform. In general
given a finite collection of distinguished open sets $U_1,...,U_r$
for $\fa$ in $M$, every point in the intersection $U_1\cap...\cap
U_r$ has a fundamental system of distinguished open sets which are
uniform with respect to each $U_j$ (cf. \cite{Go}, Lemma 1.2, page
71).

 A locally finite
open covering $\mathcal U=\{U_j\}_{j\in J}$ of $M$ is {\it
regular} for $\fa$ if: (1) each open set $U_j$ is distinguished
for $\fa$; and (2) any two or three open subsets of $\mathcal U$
having a connected intersection are uniform with respect to a same
distinguished open subset for $\fa$. In particular, (3) each
plaque of an open subset in $\mathcal U$ meets at most one plaque
of another  open set in $\mathcal U$. Every open cover of $M$ can
be refined into a regular cover (\cite{Go} Proposition 1.6, page
73).

A {\it chain}  of $\mathcal U$ is a finite
collection $\mathcal C=\{U_1,...,U_r\}$ of open subsets in
$\mathcal U$  such that two consecutive elements have non-empty
intersection. The chain $\mathcal C$ is {\it closed} if $U_r=U_1$.

Let now $\mathcal U=\{U_j\}_{j\in J}$ be a regular covering of $M$
with respect to $\fa$ and for each index $j\in J$ denote by
$\Sigma_j$ the  space of leaves of $\fa\big|_{U_j}$ with projection
$\pi_j \colon U_j \to \Sigma_j$. The foliation charts $\vr_j
\colon U_j \to \mathbb R^m=\mathbb R^l \times \mathbb R^k$ allow us
to identify each  space $\Sigma_j$ with a section (a disc)  transverse  to
$\fa$ in the chart $U_j$. By the uniformity of the open sets in $\U$,
 if $U_i \cap U_j\ne \emptyset$ then there is a local
 diffeomorphism $h_{ij}\colon \Sigma_i \to \Sigma_j$ such that
 $\pi_j=h_{ij}\colon \pi_i$ on $U_i \cap U_j$;  we also have $h_{ji}=h_{ij}^{-1}$
 and on each non-empty intersection $\pi_i(U_i\cap U_j \cap U_u)$ we have
 $h_{uj}\circ h_{ji}=h_{ui}$. The collection $\mathcal H(\fa)$ of local
 diffeomorphisms $h_{ij}$ defines the {\it holonomy pseudo-group}
 of $\fa$ with respect to the regular covering $\U$. By the above
 properties of regular coverings, this holonomy pseudogroup is
 intrinsically defined by the foliation $\fa$ and its
 localization to a  leaf $L$ of $\fa$ gives the {\it holonomy group}
 of the leaf $L$.

 The
 result below comes from the proof of the Complete  Stability
Theorem of Reeb (cf. \cite{Go}):

\begin{Proposition}
\label{Proposition:reebstability} Let $\fa$ be a transversely
oriented, codimension one, nonsingular closed foliation on a
connected manifold $T$, not necessarily compact.
\begin{itemize}
\item[{\rm (i)}] Let $L$ be a compact leaf of $\fa$ and let $L_n$
be a sequence of compact leaves of $\fa$ accumulating to $L$. Then
given a  neighborhood $W$ of $L$ in $T$ one has $L_n\subset W$ for
all $n$ sufficiently  large.

\item[{\rm (ii)}] Assume that $\fa$ has  a compact leaf with
trivial holonomy and let $\Omega(\fa)$ be the set of compact
leaves $L\in\fa$ with trivial holonomy. Then $\Omega(\fa)$ is open
in $T$ and $\partial\Omega(\fa)$ contains no compact leaf. Indeed,
a compact leaf which is a limit of compact leaves with trivial
holonomy also has trivial holonomy.

\item[{\rm (iii)}] Let $L$ be a compact leaf with finite holonomy
group. Then the holonomy of $L$ is trivial and there is a
fundamental system of invariant neighborhoods $W$ of $L$ such that
$\fa\big|_W$ is equivalent to the product foliation on $L\times
(-1,1)$ with leaves $L\times \{t\}$.

\end{itemize}
\end{Proposition}

\subsection{Holonomy of a component of the singular set}\label{Section:holonomy-singular-set}
We consider again a Bott-Morse foliation $\fa$.
We recall the  notion, introduced in \cite{Se-Sc},  of {\it
holonomy} of a component $N$ of the singular set of $\fa$ of
dimension $n \ge 0$. Consider a finite open cover $\U =
\{U_1,...,U_\ell, U_{\ell+1}\}$ of $N$ by open subsets $U_j
\subset M$ with $U_{\ell+1}=U_1$ and charts $\vr_j\colon U_j \to
\vr_j(U_j) \subset \re^m$ with the following properties:

(1) Each $\vr_j\colon U_j \to \vr_j(U_j) \subset \re^m$ defines a
local product trivialization of $\fa$; \, $\vr_j(U_j \cap N)$ is
an $n$-disc $D_j$ and $\vr_j(U_j)$ is the product of  $D_j$ by an
$m-n$ disc.

(2) $U_j \cap U_{j+1} \ne \emptyset, \forall j=1,...,\ell$.

(3) If $U_i \cap U_j \ne \emptyset$ then there exists an open
subset $U_{ij} \subset M$ containing $U_i \cup U_j$ and a chart
$\vr_{ij}\colon U_{ij} \to \vr_{ij}(U_{ij}) \subset \re^m$ of $M$,
such that $\vr_{ij}$ defines a product structure for $\fa$ in
$U_{ij}$ and $U_{ij} \cap N \supset [(U_i \cup U_j) \cap N] \ne
\emptyset$.
 In each $U_j$ we choose a transverse disc
$\Sigma_j$\,, \, $\Sigma_j \cap N = \{q_j\}$ such that
$\Sigma_{j+1} \subset U_j \cap U_{j+1}$ if $j \in
\{1,\dots,\ell\}$.

In each $U_j$ the foliation is given by a smooth function $F_j
\colon U_j \to \re$ which is the natural trivial extension of its
restriction to  the transverse disc $\Sigma_j$. There is a $C^\infty$ local diffeomorphism
$\vr_j\colon (\re,0) \to (\re,0)$ such that
$F_{j+1}\big|_{\Sigma_{j+1}} = \vr_j \circ
F_j\big|_{\Sigma_{j+1}}$\,. This implies that $F_{j+1} =
\vr_j\circ F_j$ in $U_j \cap U_{j+1}$. Notice that, as in the
classical case of non-singular foliations (see \cite{Camacho-Lins
Neto} chapter II or \cite{Go} Definition 1.5 page 72),  by
condition (3), if $U_i \cap U_k \ne \emptyset$, then the existence
of the maps $\vr_{ij}$ grants that every plaque $\fa$ in $U_i
\backslash N$ intersects at most one plaque of $U_k\backslash N$.
The {\it holonomy map associated to $N$ \/}
 is the local diffeomorphism $\vr\colon
(\re,0) \to (\re,0)$ defined by the composition $\vr = \vr_\ell
\circ\cdots\circ \vr_1$\,. This map is well-defined up to
conjugacy in $\Diff(\re,0)$.

\subsection{Holonomy of  invariant subsets of codimension one}

 We  now extend the notion of holonomy to
 connected invariant subsets $\Lambda$ of codimension one of a
 Bott-Morse foliation $\fa$. If $\Lambda$ is a compact
leaf, this only means its holonomy as a leaf of the restriction
$\fa_0$ of $\fa$ to $M\setminus \sing\fa$, and if $\Lambda$ is a component of $sing(\fa)$, then its holonomy was introduced above.
The new case is when
  $\Lambda$ is the union of a saddle-type singular component (of
arbitrary dimension $\ge 0$) with some of its separatrices.

 Let us assume  first that $N_0=\{p\}$ is an isolated saddle and
 $\Lambda = N_0 \cup \tau$, where $\tau$ is union of separatrices;
we follow \cite{Camacho-Scardua}.  Notice that in a small
neighborhood of $p$, $\tau$ can consist of one or two components
$\tau_1$ and $\tau_2$, and that this can only happen if $p$ has  Morse index $1$  or $m-1$. In this case $\Lambda$ locally divides the manifold $M$ into
three connected components. One of them, say $R_3$, is the union
of (regular) leaves which are hyperboloids of one sheet, and the
others, say $R_1$ and $R_2$, are   union of one  connected
components of hyperboloids of two sheets, as depicted in Figure
~\ref{Figure:holonomy}. Let $\gamma\colon [0,1] \to \Lambda$ be a
piecewise smooth  path on $\Lambda$  which passes through the
singularity $p$, going from $\tau_1$ to $\tau_2$.  Fix a
neighborhood  $U$ of $p\in \sing\fa$ where $\fa$ is given by a
Morse function $f$ with a unique singularity at $p$. Using the
level sets of $f$, the holonomy along $\gamma$ can be defined in
the usual way on $R_3$,  by lifting paths to the leaves. Let us
extend this map to the other regions  $R_1$ and $R_2$.    Let
$T_0$ and $T_1$ be local transverse sections to $\fa$ at $\gamma(0)$ and
$\gamma(1)$ respectively. The {\it holonomy} along $\gamma$ is the
map which carries  $t\in T_0$ to $f^{-1}(f(t))\cap T_1 \in T_1$.
This holonomy map is well-defined even if $\gamma$ is not
contained in $\{p\}\cup \tau_1$.

The extension of this concept to the case when the isolated saddle
has Morse index different from $1$ and $ m-1$ is just as in the
case of the region $R_3$ above, so we leave it to the reader.

\begin{figure}[ht]
\begin{center}
\includegraphics[scale=0.50]{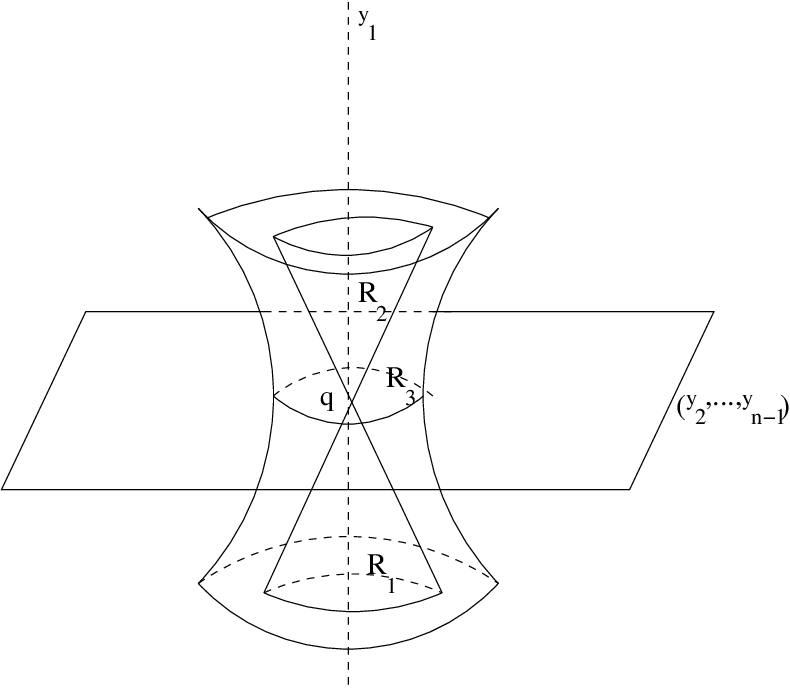}
\end{center}
\caption{Holonomy of an isolated saddle} \label{Figure:holonomy}
\end{figure}

Now we consider the case when $N_0$ is non-isolated of dimension $n_0 \geq
1$. Let $\gamma\colon [0,1] \to \Lambda$ be a path for which we
wish to define the holonomy map.
By composition with maps
already obtained as in Section~\ref{Section:holonomy-singular-set},
we can assume that the image of $\gamma$ is contained in a
neighborhood $U$ of some point $p \in N_0$, which is diffeomorphic
to a product $D_0 \times V$, where $D_0 \subset \mathbb R^{n_0}$ is
the unit disc centered at the origin and $V\subset \mathbb
R^{m-n_0}$ gives the transverse model of $\fa$ along $N_0$.
 In other words, in $U$ the foliation is
given by a Bott-Morse function $f\colon U \to \mathbb R$ of
product type: there are local coordinates $(x,y)\in U
\cong D_0 \times V$ such that $f(x,y)=g(y)$ where $g(y)$ is a
Morse function which describes the transverse type of $\fa$ along
$N_0$. Then we can
introduce the holonomy map $h$ associated to the path
$\gamma\subset U$  using exactly the same construction as above,  by
setting $h\colon T_0 \to T_1$ to be the map
which carries $t\in T_0$ into $f^{-1}(f(t))\cap T_1 \in T_1$.

\medskip

The following result is used in the sequel. When the invariant set
is a compact leaf, this is the classical Reeb local stability
theorem (\cite{Go}, \cite{Camacho-Lins Neto}), extended in
\cite{Se-Sc} to the case of centers.

\begin{Theorem}\label{Theorem:localproduct}
\label{Lemma:trivialholonomy} Let $\fa$ be a closed Bott-Morse foliation on $M$ and let $\Lambda\subset M$ be a compact
leaf, a component of the singular set or the union of a saddle
component $N\subset \sing(\fa)$ with some of its separatrices.
Then:

\begin{enumerate}

\item  The holonomy group of  $\Lambda$ is trivial.

\item  There is a fundamental system of invariant neighborhoods
$W_\alpha$ of $\Lambda$ in $M$ such that on each $W_\alpha$ the
foliation $\fa\big|_{W_\alpha}$ is given by a Bott-Morse function
$f_\alpha\colon W_\alpha \to \re$.
\end{enumerate}
\end{Theorem}

\begin{proof}
In all the cases, since $\fa$ is closed, any leaf close enough to
$\Lambda$ but not contained in $\Lambda$, is compact. Therefore
  a  holonomy map $h$  corresponds to a local
diffeomorphism $h\in \Diff^\infty(\mathbb R,0)$ such that every
orbit of $h$ is finite. Since $h$ is orientation preserving this
implies (as in \cite{Camacho-Lins Neto} Lemma 5 page 72) that
$h=Id$, proving the first statement. To prove the second statement, assume first that $\Lambda = L$ is a compact
leaf. Since it has trivial holonomy, the classical local
stability theorem implies that $L$ admits a fundamental system of
invariant neighborhoods $W$ where $\fa$ is a trivial product
foliation,  given by a submersion, as claimed. If we now assume
that $\Lambda$ is a center type component in $\sing(\fa)$, then the claim that
there is a
fundamental system of invariant neighborhoods of $\Lambda$ where
$\fa$ is equivalent to a fibre bundle over $\Lambda$
follows from Theorem
B (actually from Lemma 2.5) in \cite{Se-Sc}. Finally,
assume that $\Lambda$ is the union of a saddle $N\subset
\sing(\fa)$  and some of its separatrices.  Because  the holonomy
of $\Lambda$ is trivial, {\it a fortiori} also the holonomy of
$N\subset \sing(\fa)$ is trivial and  we apply Lemma 2.5 in
\cite{Se-Sc} to obtain a Bott-Morse function $f_0$ which defines
the foliation in a small neighborhood $U(N)$ of $N$. Again because
the holonomy of $\Lambda$ is trivial,  classical holonomy
extension arguments allow us to extend $f_0$ as a first integral to
$\fa$ in a neighborhood $W$ of $\Lambda$, constructed as the
saturated of $U(N)$.
\end{proof}

\section{Basins of Centers}\label{Section:Basins}

In this section  we look at the topology of the foliation near a
center-type component.

\subsection{Stability of foliations with center-type singularities}

We recall first  the main results of \cite {Se-Sc} that we use in the sequel.

\begin{Definition} {\rm Let $\fa$ be a possibly singular foliation on $M$. A
subset $B \subset M$ invariant by $\fa$ is {\it stable\/} (for
$\fa$) if for any given neighborhood $W$ of $B$ in $M$ there
exists a neighborhood $W' \subset W$ of $B$  such that every leaf
of $\fa$ intersecting $W'$ is contained in $W$. }
\end{Definition}

The following is essentially Proposition 2.7 in \cite{Se-Sc}:

\begin{Proposition}
\label{Proposition:Stabilitycharacterization}   Let $\fa$ be a
  Bott-Morse foliation on $M$. Given a
 compact component $N \subset \sing(\fa)$ we have:
\begin{enumerate}

\item  If $N$ is of center type and it is a limit of compact
leaves then $N$ is stable.

\item  If $\fa$ is compact in a neighborhood of $N$, then $N$ is
stable of center type with trivial holonomy.

\item  If $N$ is of center type and the  holonomy group of  $N$ is
finite then $N$ is stable and the nearby leaves are all compact.

\end{enumerate}

\end{Proposition}

 One has the following  the Local Stability Theorem in
\cite{Se-Sc} (Theorem B):

\begin{Theorem}
\label{Theorem:LocalStability}  Let $\fa$ be  a closed
Bott-Morse foliation on $M^m$ and let $N^n \subset \sing(\fa)$ be
   a center type component. Then $N$ is
stable and there is a fundamental system of invariant compact
neighborhoods $\{W_\nu\}$ of $N$ such that every leaf in $W_\nu$
is compact, with trivial holonomy and diffeomorphic to the unit
sphere normal bundle of $N$ in $M$. Hence every such leaf is an  $(m-n-1)$-sphere
bundle over $N$.
\end{Theorem}

And one also has the  corresponding  Complete Stability Theorem
(Theorem A in \cite{Se-Sc}):

\begin{Theorem}\label{Theorem:CompleteStability}
  {\sl Let $\fa$
be a smooth foliation with Bott-Morse singularities on a closed
oriented manifold $M$ of dimension $m \ge 3$ having only center
type components in $\sing(\fa)$. Assume that  $\fa$ has some
compact leaf $L_o$ with finite fundamental group, or there is a
codimension $\geq 3$ component $N$ with finite fundamental group.
Then  all leaves of $\fa$ are compact, stable, with finite
fundamental group. If, moreover, $\fa$ is transversally
orientable, then $\sing(\fa)$ has exactly two components and there
is a   differentiable Bott-Morse function $f\colon M \to [0,1]$
whose critical values are $\{0,1\}$ and such that
$f\big|_{M\setminus\sing(\fa)}\colon {M\setminus\sing(\fa)} \to
(0,1)$ is a fiber bundle with fibers the leaves of $\fa$.}
\end{Theorem}

This theorem and its proof   lead to the following
generalization of \cite[Theorem 1.5]{LSV}, which provides a
complete topological characterization of foliated manifolds admitting compact
Bott-Morse foliations:

\begin{Theorem}[\cite{Se-Sc} Theorem C, page 191]\label{Theorem C}
Let $\fa$ be a  transversally oriented,
compact foliation  with Bott-Morse singularities on a closed,
oriented, connected manifold $M^m$, $m \ge 3$, with non-empty
singular set $\sing(\fa)$. Let $L$ be any leaf of  $\fa$. Then
$\sing(\fa)$ has two connected components \,$N_1, N_2$, both of
center type, and one has:
\begin{itemize}

\item[{\rm (i)}]  $M \setminus (N_1 \cup N_2)$ is diffeomorphic to
the cylinder $L \times (0,1)$.

\item[{\rm (ii)}] $L$ is a sphere fiber bundle over both manifolds
$N_1, N_2$ and $M$ is diffeomorphic to the union of the
corresponding disc bundles  over $N_1, N_2$, glued together along
their common boundary $L$ by some diffeomorphism $L \to L$.

\item[{\rm (iii)}] In fact one has a double-fibration
%\begin{equation}
$$N_1 \buildrel{\pi_1}\over{\longleftarrow} L
\buildrel{\pi_1}\over{\longrightarrow} N_2\;,$$
%\end{equation}
and $M$ is homeomorphic to the corresponding double mapping cylinder,
{\it i.e.}, to the quotient space of $(L \times [0,1]) \bigcup
(N_1 \cup N_2) $ by the identifications $(x,0) \sim \pi_1(x)$ and
 $(x,1) \sim \pi_2(x)$.
\end{itemize}
\end{Theorem}

This yields to an explicit description
 of this type of foliations on manifolds of dimensions $3$ (and
 $4$), that will be used later in this article:

\begin{Theorem}[\cite{Se-Sc}, Theorem D]\label{Theorem D}
   Let $M$ be a closed
oriented connected $3$-manifold equipped with a transversely
oriented compact   foliation $\fa$ with Bott-Morse singularities.
Then either $\sing(\fa)$ consists of two points, the leaves are
$2$-spheres and $M$ is $S^3$, or $\sing(\fa)$ consists of two
circles, the leaves are tori and $M$ is homeomorphic to the
$3$-sphere $S^3$, a Lens space or a product $S^2 \times S^1$.
\end{Theorem}

The $3$-manifolds that appear in this theorem are exactly those admitting a Heegard splitting of genus 1 (see for instance
\cite{Hempel}).

\subsection{Topology of the basin of a center}

The previous results, particularly Theorem
\ref{Theorem:LocalStability},  motivate the following definition,
which is   one of the main concepts in this article:

 \begin{Definition}[Basin of a center]
 \label{Definition:basin}
 {\rm Let $\fa$ be a foliation with Bott-Morse singularities on $M$. We define  the set $\C(\fa)\subset M$  as  the union  of all the
 centers of $\fa$ and all the compact leaves $L \in \fa$ which
 bound a  compact invariant region $R(L,N)$, neighborhood of
 some center type component
$N\subset \sing(\fa)$, of dimension $n \ge 0$,   with the following properties:

\medskip

\noindent 1) The region $R(L,N)$ is equivalent to a fibre bundle
with fibre the closed disc $\ov D^{m-n}$ over $N$, the fibers
being transversal to the leaves of $\fa$ in $R(L,N)$.

\smallskip
\noindent 2)  Each leaf $L\subset R(L,N)$  is a fibre bundle over
$N$ with fiber the sphere $S^{m-n-1}$.

\medskip \noindent Given a center type component $N\subset
\sing(\fa)$ the {\it basin} of $N$, denoted $\C(N)=\C(N,\fa)$, is
the connected component  of $\C(\fa)$ that contains the center
$N$.}
 \end{Definition}

\medskip

We have:

%%%%%%%%%%%%%%%%%%%%%%%%%%%%%%%%%%%%%%%%%%%%%%%

\begin{Theorem}\label{Lemma:basinopen}
 Let $\fa$ be a closed Bott-Morse foliation on $M$ and $N_1, N_2$ center
 type components of the singular set of $\fa$. Then the basins
 $\C(N_i,\fa)$ are open sets in $M$, either disjoint or identical, {\it i.e.}, $\C(N_1,\fa) \cap
\C(N_2,\fa) = \emptyset$ or $\C(N_1,\fa) = \C(N_2,\fa)$, and we
have:

\begin{itemize}
\item[{\rm 1.}]
  If the boundary $\partial \C(N_1,\fa)$ is
empty, then $\C(N_1,\fa) = M$. In this case the singular set of
$\fa$ consists of exactly two center type components, say $N_1,
N_2$ (there are no saddles), the foliation is compact, given by a
Bott-Morse function $f: M \to [0,1] \subset \mathbb R$, and each
leaf is diffeomorphic to the boundary of a tubular neighborhood of
both manifolds $N_1, N_2$, so it is a fibre bundle over both $N_1$
and $N_2$ with fibre a sphere of the appropriate dimension.

\item[{\rm 2.}] If $\partial\mathcal C(N_1,\fa) \ne
 \emptyset$,
 then $\C(N_1,\fa)$ is
  diffeomorphic to the total space of the normal bundle of $N_1$ in
  $M$, and there is exactly one saddle component $N_0$ of
$\sing(\fa)$ in $\partial \C (N,\fa)$. In this case   $\partial
\mathcal C(N,\fa)$ is the union of $N_0$ and separatrices of
$N_0$.

\item[\rm 3.] If $\C(N_1,\fa) \ne \C(N_2,\fa)$ and  $\partial
\C(N_1,\fa) = \partial \C(N_2,\fa)$, then $M = \ov{\C(N_1,\fa)}
\cup \ov{\C(N_2,\fa)}$.

\item[\rm 4.]
 If $\partial \C(N_1,\fa) \ne
\partial \C(N_2,\fa)$ and $\partial \C(N_1,\fa) \cap
\partial \C(N_2,\fa) \ne \emptyset$,
then there is a saddle component  $N_0 \subset \partial
\C(N_1,\fa) \cap
\partial \C(N_2,\fa)$  of Morse index $1$ or $m-n_o-1$, where $n_0$ is
the dimension of $N_0$.
\end{itemize}
\end{Theorem}

\begin{proof} First notice that Theorems \ref{Theorem:LocalStability} and
 \ref{Lemma:trivialholonomy}
 imply that the sets $\C(N_i,\fa)$ are
non-empty. That they are open sets in $M$, and either they are
disjoint or identical is immediate from the definition of the
basin and Reeb's Local Stability Theorem (see
Theorem~\ref{Theorem:localproduct}).

To prove the statement (1) notice  that if  the open
subset $\C(N_1,\fa)$ of $M$ has empty boundary then it is also closed
and therefore $\C(N_1,\fa) = M$, by connectedness.  Hence the
foliation is compact and there are no saddles. Then the rest of
statement (1) follows from Theorem \ref{Theorem C}.

Let us prove statement (2). For this we must show:

\begin{Claim} If $\partial \C(N_1,\fa) \ne
\emptyset$ then
 $\C(N_1,\fa)$
 is the union of $N_1$ and all the compact leaves that bound a
tubular neighborhood of $N_1$.
\end{Claim}

This obviously implies that $\C(N_1,\fa)$ is diffeomorphic to the
normal bundle of $N_1$ in $M$.

\begin{proof}[Proof of the claim]
For a leaf $L\subset \partial C (N_1,\fa)$ close enough to
$N_1$ it is clear that $L$ bounds a tubular neighborhood $R(L)$ of
$N_1$ in $M$. Applying the local stability theorem of Reeb to $L$
this same property holds for any leaf $L_1$ close enough to $L$.
Thus, by the connectedness of $\partial C(N_1,\fa)$ it follows that any
leaf in $\partial C(N_1,\fa)$ bounds a tubular neighborhood of
$N_1$.  The same argumentation actually shows that $
C(N_1,\fa)$ is the union of $N_1$ and all compact leaves that bound
a tubular neighborhood of $N_1$,  with the projection having as fibre a disc
 transverse to the
leaves.
\end{proof}

 The above  arguments prove the first claim in (2). Let us
prove now that  there is no compact
leaf in $\partial \C(N_1,\fa)$. If $L\subset \partial \C(N_1,\fa)$
is compact and  accumulated by a sequence of leaves $L_\nu\subset
\C(N_1,\fa), \, \nu \in \mathbb N$,
 then given $W$ as in Theorem~\ref{Theorem:localproduct}
we have $L_\nu\subset W$ for $\nu
>>1$ and since $\fa$ is of product type in $W$ we have $W\subset
\C(N_1,\fa)$ so that $L\not\subset \partial \C(N_1,\fa)$, a
contradiction. Therefore every leaf $L\subset \partial
\C(N_1,\fa)$ is separatrix of some saddle component. By
definition, if a center component $\tilde N$ is accumulated by
leaves in $\C(N_1,\fa)$ then $\tilde N$ is contained in
$\C(N_1,\fa)$. Hence there are no centers in $\partial \C(N,\fa)$.

 On the other hand, given a
saddle component $N_0\subset \partial \C(N_1,\fa)$, it is clear
that some separatrix of $N_0$ is contained in $\partial
\C(N_1,\fa)$. Let us prove that there is exactly one saddle
component in $\partial \C(N_1,\fa)$. If $N, N_0$ are different
saddles in $\partial \C((N_1,\fa)$ then there is a sequence of
compact leaves $L_\nu \subset \C((N_1,\fa)$, $\nu \in \mathbb N$,
which accumulate both $N$ and $N_0$ as $\nu \to \infty$. Hence
there exist separatrices ${\cal L}$ and ${\cal L}_0$ of $N$ and
$N_0$, respectively, which are accumulated by the $L_\nu$. Notice
that the sets $\Lambda= {\cal L} \cup N$ and $\Lambda_0 = {\cal
L}_0 \cup N_0$ are both compact and invariant, so by
Theorem~\ref{Lemma:trivialholonomy} they have trivial holonomy and
  each of these sets has a
 fundamental system $W_\nu$, $W_{\nu_0}$ of invariant neighborhoods.
  If ${\cal L} \ne {\cal L}_0$ then we have
$\Lambda\cap \Lambda_0=\emptyset$ and therefore $W_\nu \cap
W_{\nu_0}=\emptyset$ for $W_\nu$, $W_{\nu_0}$ small enough. On the
other hand we have $L_\nu \subset W_\nu$ and $L_\nu \subset
W_{\nu_0}$ for all $\nu,\nu_0  >>1$, a contradiction, since by
hypothesis there are no saddle connections. This proves (2).

For (3) we notice that if $\C(N_1,\fa) \ne \C(N_2,\fa)$ and
$\partial \C(N_1,\fa) = \partial \C(N_2,\fa)$, then $
\ov{\C(N_1,\fa)} \cup \ov{\C(N_2,\fa)}$ is open and obviously
closed in $M$, so the statement follows by connectedness.

Finally we prove (4). Suppose that $\partial \C(N_1,\fa) \ne
\partial \C(N_2,\fa)$ and $\partial \C(N_1,\fa) \cap
\partial \C(N_2,\fa) \ne \emptyset$.   By (ii)
there is a single saddle component $N_0\subset \po\C(N_1,\fa) \cap
\po\C(N_2,\fa)$.
 If the transverse Morse index of $N_0$ is different from $1$
and $m-n_0- 1$, then in  suitable local coordinates
$(x_1,...,x_{n_0}, y_1,\dots,y_{m-{n_0}})\in M$ we have $N_0 =
\{y_1=...=y_{m-{n_0}}=0\}$ and the union  $\Lambda(N_0)$ of $N_0$
and the local separatrix through $N_0$ is given by $y_1^2 +\cdots+
y_r^2 = y_{r+1}^2+\cdots+ y_{{m-n_0}}^2 $ where $r \notin \{1,
m-n_0-1\}$. Hence the local separatrix is connected. This implies
$N_0$ has only one separatrix and therefore
$\po\C(N_1,\fa) = \po\C(N_2,\fa)$, by (ii), which is a contradiction. Hence
the transverse Morse index of $N_0$ must be as stated  in (4).
\end{proof}

In the  situation envisaged in (2) and  (4) we  say that $N_0$ and
$N_1$ are {\it paired} (cf.
Definition~\ref{Definition:pairing}).

\medskip
Since every circle in an oriented manifold has trivial normal
bundle, Theorem~\ref{Lemma:basinopen} implies:

\begin{Corollary}\label{Lemma:basinopen-dim3} {\rm Let $\fa$ be a foliation with Bott-Morse singularities on $M$ and let $N_0\subset \sing(\fa)$ be a center type singularity. If $N_0$ is an isolated singularity, then the leaves around it are
$(m-1)$-spheres, and if $\partial C(N_0,\fa) \ne \emptyset$ then
the interior of $C(N_0,\fa)$ is an $m$-ball $D^{m}$. Similarly
if $N_0$ has dimension $1$ then every nearby leaf is
diffeomorphic to $S^1 \times S^{m-2}$ and if $\partial \C(N_0,\fa)
\ne \emptyset$ then the interior of $\C(N_0,\fa)$ is a product
$S^1 \times D^{m-1}$.}

\end{Corollary}

\begin{Remark}\label{Remark:analytic}
 {\rm The above Corollary~\ref{Lemma:basinopen-dim3} also shows
 that if $\fa$ is a (transversely) real analytic
Bott-Morse foliation on a connected closed oriented manifold $M^n, n
\ge 3$, such that $\fa$ has a compact leaf with finite holonomy then
$\fa$ is closed. This occurs, for instance, if $\sing(\fa)$ has an
isolated center or, more generally, some center component $N$ with
$|\pi_1(N)|<\infty$ and $\codim \geq 3$ (cf. \cite[Stability
theorem]{Se-Sc}).  To see this, let $L\in \fa$ be a compact leaf
with finite holonomy, denote by $\Omega$ the union of all compact
leaves of $\fa$ in $M$ and let $\Omega(L)\subset\Omega$ be its
connected component containing $L$. As we know already, if a leaf
$L_1\in \fa$ is such that $L_1$ does not accumulate on a component
of $\sing(\fa)$ and it is in the boundary of $\Omega(L)$, then $L_1$
is compact. The holonomy group  of $L_1$ is a subgroup of real
analytic diffeomorphisms $\Hol(L_1,\fa) < \Diff ^w (\re,0)$. Since
$L_1$ is a limit of
 compact leaves, the holonomy group $\Hol(L_1,\fa)$ has  finite
orbits arbitrarily close to the origin. This implies that
$\Hol(L_1,\fa)$ is finite and therefore trivial. Since $L_1$ is
compact and has trivial holonomy it follows from the local
stability theorem that we actually have $L_1\subset \Omega(L)$, a
contradiction. Thence every leaf in the boundary $\partial
\Omega(L)$ must accumulate on $\sing(\fa)$ and therefore it has to
be a separatrix of some saddle singularity. Also this leaf is
closed off $\sing(\fa)$. Thus, every   compact invariant subset
$\Lambda(N)\subset \partial \Omega(L)$, obtained as the union of a
saddle  $N $ and some of its separatrices, must have trivial
holonomy. Therefore all leaves near $\Lambda(N)$ are compact. This
shows that $\partial \Omega$ is a union of sets of the form
$\Lambda(N)$ as above. Since $M$ is connected we  conclude that
$M=\Omega\cup \partial \Omega$, which implies that $\fa$ is a
closed foliation.}
\end{Remark}

\section{Distinguished neighborhoods of saddles}
\label{Section:distinguishedneigh}
 We now look at the topology of the foliation near a
saddle-type component of the singular set. We discuss first the
isolated singularity case since this gives the local model. We
begin with a fast review of classical Morse theory. We consider a foliation $\fa$ with Bott-Morse
singularities on a manifold $M$.

 \vglue.1in

\subsection{Isolated saddle}
Let us assume that near a saddle point $p \in M$ the foliation is
given by the fibers of a Morse function $f \colon U \to \re$ with
Morse index $r$ at $x$. That is, we can choose local coordinates
where $x$ is identified with the origin $0 \in \re^m$ and $f$ is:

\begin{equation}\label{eq. Morse-saddle}
f(x_1,...,x_m) =
\sum\limits_{j=1}^r - x_j^2 + \sum\limits_{j=r+1}^m x_j^2\quad,\qquad  m > r
> 0\,.
\end{equation}
 Notice that the local separatrix of $\fa$ union the singular point $0$
is the hypersurface $V$ given locally by $f^{-1}(0)$:

$$ V \,=\, \big \{
\sum\limits_{j=1}^{r}  \,x_j^2 \,=  \sum\limits_{j=r+1}^{m} x_j^2
\big \} \,.$$ This is a cone, union of lines passing through the
origin. Hence its topology is determined by its {\it link} $K = V
\cap S_\epsilon$ where $S_\epsilon$ is a sphere around $0$ (see
\cite{Milnor3}), which may actually be taken to be the unit
sphere, that we denote $\mathbb S$ and we think of it as bounding
the unit ball $\mathbb D$. It is an exercise to show that $K$ is
diffeomorphic to the product $S^{r-1} \times S^{m-r-1}$ and
therefore the separatrix $V \setminus\{0\}$ is $S^{r-1} \times
S^{m-r-1} \times \re$ (in fact its intersection with $U$, assuming this set is small enough).

Notice that if $r$ is $1$ or $m-1$, and only in these cases,
 the link $K$ has two connected components, otherwise it is
 connected. This is because in these two cases $V$ actually has two
 ``branches'' (or components) that meet at $0$. This means that if $r=1,n-1$, then
 the saddle has two local separatrices. Thus in these
 cases, and only in these cases, an isolated saddle can have
 either one or two global separatrices, since both local
 separatrices can belong to the same global leaf.

The hypersurface $V$ splits $\re^m$ in two regions, corresponding to the points $x \in
U$ where $f(x)$ is positive or negative (for $r = 1, m-1$ one actually
has three regions, an ``external one'' and two ``internal regions''
bounded by the two components of $V$). Let us look at the fibers
$V_t = f^{-1}(t)$ of $f$ in these two cases.

We observe that $V$ meets the unit sphere  $\mathbb S$
transversally. Thus, essentially by Thom's transversality, for $t
\ne 0$ sufficiently small one has that $V_t$ also meets  $\mathbb
S$ transversally. Then the first Thom-Mather Isotopy Theorem
implies that  for $t \ne 0$ sufficiently small, $V$ and $V_t$ are
isotopic away from  $\mathbb D$, so they only differ inside the
ball; the $V_t's$ also differ only inside the ball. In fact, to
get $V_t$ all one has to do, up to diffeomorphism, is to take  the
piece of $V$ that is outside the interior of the unit ball, $V^* =
V \setminus \buildrel{\circ}\over{\mathbb
  D}$, which has boundary $K =S^{r-1} \times S^{m-r-1}$,  and attach to
it either $D^r  \times S^{m-r-1}$ to get the fibers $V_t$ for $t > 0$,
or $S^{r-1} \times D^{m-r}$ to get the fibers $V_t$ for $t < 0$.

In other words, if $t^+$ is positive and $t^-$ is negative, then
the fibers  $V_{t^+}$ and  $V_{t^-}$ are obtained from each other
by the classical Milnor-Wallace surgery (see \cite {Milnor1}),
{\it i.e.}, by removing $D^r  \times S^{m-r-1}$ from  $V_{t^+}$ to
get a manifold with boundary  $S^{r-1} \times S^{m-r-1}$, and then
attaching   $S^{r-1} \times D^{m-r}$ to get
 $V_{t^-}$, and viceversa.

\medskip
In the sequel we  need to consider  {\it distinguished
  neighborhoods}
of  saddle singularities. This means a set $W(N_0)$, homeomorphic to an $m$-ball, of the form:
$$W(N_0)\,=\,\{(x_1,...,x_m) \in  B_{\epsilon} \; \big | \, -\delta \, \leq\; f(x_1,...,x_m)
 \,\leq \delta\}\,,
$$
where $ B_{\epsilon} \subset M$ is diffeomorphic to a ball.
Thus  $W(N_0)$ is bounded by the leaves $f^{-1}(\pm \delta)$ and the
sphere  $S_{\epsilon} = \partial B_{\epsilon}$, with $\delta$
small enough with respect to $\epsilon$,  so that all the fibers $f^{-1}(t)$ with $|t| \le
\delta$ meet  $ \partial B_{\epsilon}$ transversally.

In other words,
 $W(N_0)$ can be regarded as a {\it Milnor tube} for $f$ at $N_0$ (see
 the last chapter of \cite {Milnor3}). Its boundary $\partial W(N_0)$
 is homeomorphic to an $(m-1)$-sphere and  consists of three
 ``pieces'': two leaves of $\fa$ and the {\it cap} denoted ${\cal C}$
 consisting of the points  $x \in S_{\delta}$ with $|f(x)| <
 \delta$:

\begin{equation}\label{decomposition disting. neigh}
\partial W(N_0) \, = \, [(f^{-1}(- \epsilon) \cup
f^{-1}(\epsilon)) \cap  B_{\delta}] \cup {\cal C}\,.
\end{equation}
At each point in ${\cal C}$ the corresponding leaf of
 $\fa$ is transversal to $\partial W(N_0)$.
Notice that we speak of three ``pieces'' in $\partial W(N_0)$,
each of which may have one or two connected components, depending
on the Morse index of the corresponding saddle.

 \vglue.1in

\subsection{Non-isolated saddle} Let  $N_0$ be now
  a  saddle singularity of  dimension $n_0>0$.
Given a point $p \in N_0$, define {\it a flow box} for $\fa$ at
$p$ to be a set of the form $W = \Sigma \times D^{n_0}$ where:

i) $\Sigma$ is an $(m-n_0)$-disc transversal to $N_0$ at $p$, with
a Morse foliation defined by the intersection of  $\Sigma$  with
the leaves of $\fa$,  singular at $p$; and

ii) restricted to $W$ the foliation $\fa$ is a product foliation.

\noindent We say that $W$ is {\it a distinguished flow box} if the
transversal $\Sigma$ is a distinguished  neighborhood of the
isolated saddle $p$.

\begin{Definition} {\rm A {\it distinguished neighborhood} of $N_0$ is a
  compact neighborhood  $W(N_0)$ of $N_0$, which is union of
 a finite collection $W_1,...,W_s$ of distinguished flow boxes for
 points in $N_0$, such that:
 \begin{enumerate}

 \item  The intersection of any two of them is
 either empty or a flow box.

 \item   $W(N_0)$ can be identified with the normal
bundle of $N_0$ and therefore $\fa|_{W(N_0)}$ has locally a
product structure with an isolated saddle singularity in each
normal fibre. Hence each normal fibre inherits a decomposition in
three pieces as in equation {\rm(\ref{decomposition disting.
neigh})}.

\item  The boundary of $W(N_0)$ is union of leaves of $\fa$ and a
fibre bundle $\widetilde {\cal C}$ over $N_0$ with fiber the cap
$\cal C$ in equation {\rm(\ref{decomposition disting. neigh})},
consisting of points where the foliation is transversal to
$\partial W(N_0)$.

 \end{enumerate} }
\end{Definition}

%\begin{Remark}
%\label{Theorem:localproduct} {\rm Given a compact leaf $L\subset
%M$, since the holonomy of $L$ is trivial,  by classical Reeb local
%stability theorem (\cite{Go}, \cite{Camacho-Lins Neto}) there is a
%fundamental system of invariant neighborhoods $W$ of $L$ in $M$
%where $\fa$ is equivalent to a product foliation. The same statement
%can not possibly be truth for the invariant sets $\Lambda$ that we
%consider in general in this section.
%For instance, as shown in \cite {Se-Sc}
%} \end{Remark}

The existence of distinguished neighbourhoods  for closed
Bott-Morse foliations is granted by the local product
structure at each component of the singular set, the triviality of the holonomy, and the compactness of the singular set.

\begin{Proposition}\label{Lemma:productN_0}
 Let $\fa$ be a closed Bott-Morse foliation on a closed manifold $M$ and let $N_0$ be  saddle-type
component of the singular set.  Let $W(N_0)$ be a distinguished
neighborhood of $N_0$. Then for each leaf $L$ of $\fa$ that meets
$W(N_0)$, one has that the intersection $L \cap W(N_0)$ is a fiber
bundle over $N_0$ with fiber $L \cap \Sigma$, the trace of $L$ in
a transversal $\Sigma$.
\end{Proposition}

 The proof is obvious.

\begin{Remark}\label{Rem:notptoducts} {\rm
It is not truth in general that the leaves $L \cap W(N_0)$ have a
global product structure, even if the normal bundle of $N_0$ is
trivial. For instance take a $2$-disc $D$ in $\re^2$ with a
saddle singularity at $0$; now take in $\re^3$ the product $D
\times [0, 2 \pi]$. As you move upwards from the level $D \times
\{0\}$ a time  $t$, rotate the disc by an angle $t$. Hence at
the level $D \times \{1\}$ the foliation on the disc is exactly as
in the level $D \times \{0\}$, so we can glue these two $2$-discs
and get a foliation on a solid torus $D \times S^1$ with trivial
holonomy;  no leaf has a global product structure. Nevertheless,
this is not surprising since already in the nonsingular case a
compact leaf with trivial holonomy or homotopy does not give a
global product structure for the foliation, but a fibre bundle
structure.}
\end{Remark}

\section{The Partial stability theorem}\label{sec. Partial Stability Theorem}

In this section we give one of the main results in this work, the
Stability Theorem~\ref{Theorem:partialstabilitytheorem}.
This is analogous to the classical Partial Stability Theorem of
Reeb for compact leaves, and to the Partial Stability Theorem
\ref{Theorem:LocalStability}
 for
center-type components of Bott-Morse foliations. Then we look at the
global topology of separatrices and  use the Stability Theorem to
describe the topology of the nearby leaves.

\subsection{The Partial stability theorem}

As in the classical case of non-singular foliations with trivial
 holonomy, using also the product structure of $\fa$
around $N_0$ we obtain:

\begin{Proposition}
\label{Proposition:trivialneighborhood} Let $\fa$ be a closed
Bott-Morse foliation on a compact manifold $M$, $N_0\subset
\sing(\fa)$ a saddle component  and  $\Lambda(N_0)$ be  the union
of $N_0$ and all its separatrices.  Consider a fundamental system
of compact invariant neighborhoods $\{W_\nu\}$ of $\Lambda(N_0)$
in $M$ as in Theorem~\ref{Lemma:trivialholonomy}, and a
distinguished neighborhood $U(N_0)$ of $N_0$. Then for each $\nu
>>1$ one has that   $W_\nu \setminus (W_\nu \cap U(N_0))$ fibers
over $\Lambda (N_0)\setminus (\Lambda(N_0)\cap U(N_0))$ by
transverse segments $\Sigma_x, \, x \in \Lambda (N_0)\setminus
U(N_0)$, so that for every leaf $L$ close enough to $\Lambda(N_0)$
the intersection $L\cap \Sigma_x$ is a single point.
\end{Proposition}

\begin{proof}
This is  essentially a consequence of the proof of Reeb's local  stability
theorem in \cite{Camacho-Lins Neto} (see Lemma 6 and Theorem 4 in
Chapter 4). By Theorem~\ref{Lemma:trivialholonomy}.
  there is a fundamental system of
distinguished neighborhoods $\{V_\alpha\}$ of $N_0$ such that
$\fa\big|_{V_\alpha}$ is given by a Bott-Morse function $f_\alpha$
with singular set $N_0$.

Denote by $\Lambda_j, \, j=1,...,r$ the connected components of
$\Lambda(N_0) \setminus N_0$. Given a distinguished neighborhood
$V=V_\alpha$ of $N_0$  we set
$\tilde\Lambda_j=\Lambda_j\setminus(\Lambda_j \cap V)$. The
$\tilde\Lambda_j$ are the  connected components of
$\Lambda(N_0)\setminus (\Lambda(N_0)\cap V)$. Set
$\tilde\fa:=\fa\big|_{M\setminus V}$. The leaves of $\tilde \fa$
are closed with trivial holonomy. Thus by the same arguments as in
the proof of Reeb's stability theorem, there is a fundamental system
of invariant neighborhoods $\{\tilde W_\nu ^{(j)}\}_\nu$ of
$\tilde \Lambda _j$ where $\tilde\fa$ is equivalent to a product
foliation on $\tilde\Lambda_j \times (-1,1)$ with leaves of the
form $\tilde\Lambda_j \times \{t\}, \, t \in (-1,1)$.

Moreover, there is a fibration of $\tilde W_\nu^{(j)}$ over
$\tilde \Lambda_j$ by transverse segments meeting  each leaf that
intersects $\tilde W_\nu ^{(j)}$ at exactly one point.
Considering the unions $W_\nu:=\big(\bigcup\limits_{j} \tilde W_\nu
^{(j)}\big)\cup V$ we obtain compact invariant neighborhoods of
$\Lambda(N_0)$ as in the statement.
\end{proof}

\begin{Definition}
\label{definition:orderone} {\rm Let $N_0\subset \Sad(\fa)$ be a
saddle component of the singular set of a Bott-Morse foliation $\fa$
and denote by  $\Lambda(N_0)$  the union of $N_0$
and its separatrices. We  say that $\fa$ has {\it order one over
$\Lambda(N_0)$} if there is a   distinguished neighborhood
$U(N_0)$ of $N_0$ and a fundamental system of compact invariant
neighborhoods $W_\nu$ of $\Lambda(N_0)$, such that:

\medskip
\noindent
 (1) For each $W_\nu$ one has that $W_\nu \setminus (W_\nu \cap U(N_0))$  fibers over
$\Lambda(N_0)\setminus (\Lambda(N_0)\cap U(N_0))$  by transverse
segments $\Sigma_x, \, x \in \Lambda (N_0)\setminus U(N_0)$; and

\smallskip \noindent
 (2) for every leaf $L$ close enough to $\Lambda(N_0)$ the
intersection $L\cap \Sigma_x$ is a single point.

\medskip \noindent
In this case each neighborhood $W_\nu$ is called a {\it bundle
neighborhood} of $\Lambda(N_0)$ with respect to $\fa$.}

\end{Definition}
The following result is essential in our work.

\begin{Theorem}[Partial Stability Theorem]
\label{Theorem:partialstabilitytheorem} Let $\fa$ be a closed
Bott-Morse foliation on a manifold $M$ and $N\subset \sing(\fa)$
 a center-type component with basin $\mathcal C(N,\fa)$. Let
$N_0\subset
\partial \mathcal C(N,\fa)$ be a saddle type component and let
$V\subset M$ be a distinguished neighborhood of $N_0$.  Let
$\Lambda(N_0)\subset \partial \C(N,\fa)$ be the union of $N_0$ and
a separatrix and $\Lambda_0$ a connected component of
$\Lambda(N_0)\setminus (\Lambda(N_0)\cap V)$. Then there is a
fundamental system of neighborhoods $\{W_\alpha\}$ of $\Lambda_0$,
each contained in a bundle neighborhood $W_\nu$,  such that if
$L_i,\, L_e$ are interior and exterior leaves of $\fa$
intersecting $W_\alpha$ with $L_i \subset \C(N,\fa)$ and $L_e\cap
\C(N,\fa)=\emptyset$, then $L_i\cap W_\alpha$ and $L_e \cap
W_\alpha$ are homeomorphic to $\Lambda_0$.
\end{Theorem}

\begin{proof}
 The proof is a straightforward consequence of Proposition~\ref{Proposition:trivialneighborhood}.
 Let $\Lambda(N_0)$ be the union of $N_0$ with all its
separatrices. Consider the maps $\phi_i\colon\Lambda_0 \to L_i, \,
x\mapsto \phi_i(x):=L_i \cap \Sigma_x$ and $\phi_e \colon
\Lambda_0 \to L_e, \, x \mapsto \phi_e(x):=L_e \cap \Sigma_e$.
Then we define $\phi \colon L_i \to L_e$ as $\phi=\phi_e \circ
\phi_i^{-1}$.
\end{proof}

%\begin{Lemma}??? \label{Lemma:boundaryintersect} Let $N_1,
%N_2 \subset \sing(\fa)$ be center-type components such that
%$\partial \mathcal C(N_1, \fa)=\partial \mathcal C(N_2, \fa)\ne
%\emptyset$. Then $M=\ov{\mathcal C(N_1, \fa)}\cup \ov{\mathcal
%C(N_2, \fa)}$ and $\sing(\fa)=N_0\cup N_1 \cup N_2$.
%\end{Lemma}

%\begin{proof}
%It is enough to observe that the set  $\overline{ \mathcal C(N,
%\fa)} \cup \overline{ \mathcal C(N_2, \fa)}$ is  closed and  open
%in $M$, for otherwise there would be a saddle-connection. Thus by
%connecteness of $M$ we have $M= \overline{ \mathcal C(N, \fa)}
%\cup \overline{ \mathcal C(N_2, \fa)}$.
%\end{proof}

%Notice that in the above case we have $c(\fa)=2, \, s(\fa)=1$ so
%that the inequality $c(\fa)> 2 s(\fa)$ is not verified.

\subsection{Topology of leaves near compact invariant sets with saddles}

As before, let $N$ be a saddle-type component of the singular set
of a Bott-Morse foliation $\fa$ on a manifold $M$, and let
$\Lambda$ be the union of $N$ and all its separatrices. We want to
use the results of the previous sections to compare the topology
of the leaves near $\Lambda$ with that of $\Lambda$ itself, so we
say first a few words about the latter. We start with the case of
an isolated saddle.

Assume $N= \{0\}$ is an isolated saddle with Morse index $r = 1$
or $r = m -1$. Since both cases are similar we discuss only the
case $r=1$. As mentioned earlier, in this case there are two local
separatrices, corresponding to the two branches of the
hypersurface
$$ V \,=\, \big \{
\,x_1^2 \,=  \sum\limits_{j=2}^{m} x_j^2 \big \} \,.$$ The saddle
is the point $0$ and the set $V \setminus \{0\}$ consists of two
connected components. It can happen that both components belong to
the same leaf or that they belong to different leaves. In the
first case $\Lambda$ consists of a single leaf $L$, compactified
by attaching to it the saddle singularity $0$; this situation is
depicted in figure \ref{Figure:isolated1} below. One has a
self-saddle-connection.

If the two components of $V \setminus \{0\}$ belong to different
leaves, then $\Lambda$ consists of these two leaves union the
saddle point $\{0\}$, and there cannot be more separatrices since
every separatrix must contain a local separatrix, and there are
only two of them.

When the Morse index of the saddle at $0$ is not $1$ nor $m-1$,
then the local separatrix is connected, so there can only be one
global separatrix $L$,  thus $\Lambda = L \cup \{0\}$.

Similar statements obviously hold for a non-isolated saddle $N$ of
dimension $n$: if we look at the restriction to a distinguished neighborhood of $N$
we see that the separatrix (or separatrices) is a fibre bundle
over $N$ with fibre the trace in a transversal. Hence, if its
(transverse) Morse index is $1$ or $m-n-1$, then one has two local
separatrices, which may belong to the same leaf of $\fa$ or to two
different leaves, and this is independent of the choice of point
in $N$. If the Morse index is not $1$ nor $m-n-1$, then one has
only one local separatrix, which must belong to a single leaf.

Summarizing one has:

\begin{Lemma}
Let $\fa$ be a closed Bott-Morse foliation on $M$.  Let $N$ be an $n$-dimensional, $n \ge 0$,
saddle-type component of the singular set of $\fa$, and let
$\Lambda$ be the union of $N$ and all its separatrices. Then:

{\bf i)} If the Morse index of $N$ is neither $1$ nor $(m-n-1)$, then
$N$ has only one local separatrix. Hence $\Lambda$ consists of $N$
and one single leaf $L$ of $\fa$: $\Lambda = L \cup \{N\}$.

{\bf ii)} If the Morse index of $N$ is $1$ or $(m-n-1)$, then $N$
has two local separatrices, which may or may not belong to the
same leaf. Hence $\Lambda$ may consist of $N$ and one single leaf
$L$, or $N$ union two leaves.
\end{Lemma}

Notice that since $\fa$ is closed, of codimension $1$
and transversally oriented,  if we denote by $\buildrel
{\circ}\over{W}(N)$ the interior of a distinguished neighborhood
of $N$, then $\Lambda \setminus \buildrel {\circ}\over{W}(N)$ is a
compact, codimension 1, oriented submanifold of $M$, consisting of
either one or two components, which are leaves of $\fa$ restricted to
${({M\setminus \buildrel {\circ}\over{W}(N)})}$.

Let us now look at the leaves of $\fa$ near $\Lambda$. Since the
foliation is closed and has Bott-Morse singularities, we can
choose a distinguished neighborhood $W(N)$ small enough so that
every leaf of $\fa$ that meets $W(N)$  is compact. This
condition is satisfied whenever $N$ is in the boundary of the
basin of some center, by Theorem~\ref{Theorem:partialstabilitytheorem}.

The case of an isolated saddle $N$ is now rather simple: assume
first $\Lambda$ consists of $N$ and one single leaf $L_0$. $L_0$
is oriented, of codimension 1, it splits $M$ in two connected
components that we denote by $M^+$ and $M^-$. By Theorem~\ref{Theorem:partialstabilitytheorem}, away from a distinguished
neighborhood $W(N)$, the leaves of $\fa$ in either side $M^+$ and
$M^-$ are homeomorphic to $\Lambda \setminus W(N)$. Then, from the
discussion for isolated saddles in Section
\ref{Section:distinguishedneigh} we see that if $L^+$ is a leaf in
$M^+$ that meets $W(N)$ and $N$ has Morse index $r$, then $L^+
\setminus W(N)$ has boundary $S^{r-1} \times S{m-r-1}$, which is
also the boundary of $\Lambda \setminus W(N)$ and the boundary of
$L^- \setminus W(N)$ for leaves in $M^-$.

To recover $L^+$ up to homeomorphism, we attach $D^{r} \times
S{m-r-1}$ to $\Lambda \setminus W(N)$ along their common boundary;
to get $L^-$ we attach $S^{r-1} \times D{m-r}$ to $\Lambda
\setminus W(N)$ along their common boundary; and to get $\Lambda$
back we simply collapse the boundary of $L^-  \setminus W(N)$ to
a point.

In other words, the leaves $L^+$ and $L^-$ are obtained from each
other by Milnor-Wallace surgery (\cite{Milnor1}).

If the Morse index $r$ of $N$ is $1$ or $m-1$, what we are doing
is that we remove from one leaf $L^+$ a copy of $S^{m-2} \times
D^1$, where $D^1 = I $ is an interval, so one has boundary
$S^{m-2} \times S^0$, {\it i.e.}, two $(m-2)$-spheres, and then
glue back a copy of $D^{m-1} \times S^0$, {\it i.e.}, two
$(m-1)$-discs, or viceversa. In other words, $L^+$ is obtained
from $L^-$ by attaching to it a 1-handle.

If the Morse index $r$ of $N$ is $1$ or $m-1$ and $\Lambda$
consists of $N$ and two leaves $L_1$ and $L_2$, then all the
previous discussion applies in exactly the same way, the only
difference is that now $L^-  \setminus W(N)$ has two connected
components, homeomorphic to $(L_1 \cup L_2) \setminus W(N)$, and
$L^+$  is the connected sum of these two manifolds.

Now consider a saddle component $N$ of dimension $n>0$ and
(transverse) Morse index $r$, $m - n > r > 0$.  Let $W(N)$ be an
open distinguished neighborhood of $N$, sufficiently small so that
every leaf that meets  $W(N)$ is compact. Such a neighborhood
exists because the foliation is proper, there are no
saddle-connections, there are only finitely many components of
the singular set of $\fa$, and at each component there are at most
two separatrices.

As in the isolated singularity case, $\Lambda$ splits
 $M$ in two connected
components that we denote by $M^+$ and $M^-$. By theorem
\ref{Theorem:partialstabilitytheorem}, away from $W(N)$, the
leaves of $\fa$ in either side $M^+$ and $M^-$ are homeomorphic to
$\Lambda \setminus W(N)$. From the discussion  in Section
\ref{Section:distinguishedneigh} we see that if $L^+$ is a leaf in
$M^+$ that meets $W(N)$, then $L^+ \setminus W(N)$ has boundary
$K$ which is diffeomorphic to an $(S^{r-1} \times
S^{m-n-r-1})$-bundle over $N$; up to homeomorphism, this  is also
the boundary of $\Lambda \setminus W(N)$,  and the boundary of
$L^- \setminus W(N)$ for leaves in $M^-$ sufficiently near
$\Lambda$.

To recover $L^+$ up to homeomorphism, we attach to $\Lambda \setminus
W(N)$  the corresponding bundle with fibre  $(D^{r} \times
S^{m-n-r-1})$, {\it i.e.}, we ``fill in'' the $(r-1)$-sphere in each fiber;
to get $L^-$ we ``fill in'' the $(m-n-r-1)$-sphere in each fiber, {\it
  i.e.},  we attach the corresponding $(S^{r-1} \times
D^{m-n-r})$-bundle
to $\Lambda \setminus W(N)$ along their common boundary; and to get $\Lambda$
back we collapse to a point each  $(S^{r-1} \times S^{m-n-r-1})$-fiber in
 the boundary of $L^-  \setminus W(N)$.

\medskip

We summarize the previous discussion in the following theorem:

\begin{Theorem}\label{topology of boundary}
 Let $\fa$ be a closed Bott-Morse foliation on $M^m$.
Let $N$ be a saddle-type component of dimension $n \ge 0$ and
{\rm(}transverse{\rm)} Morse index $r$, $m-n > r > 0$. Let  $W(N)$
be a sufficiently small distinguished open neighborhood of $N$
such that every leaf of $\fa$ that meets  $W(N)$ is compact. Let
$\Lambda$ be the union of $N$ and its separatrices. Then:

{\bf i)} The set $\Lambda$ consists of $N$ and at most two leaves
of $\fa$. Moreover, if $r \ne 1, m-n-1$, then there is exactly one
leaf of $\fa$ in $\Lambda$.

{\bf ii)} If $L$ is a leaf of $\Lambda$ sufficiently near
$\Lambda$ at some point, then $L \setminus W(N) $ is a compact
manifold with boundary $K$ homeomorphic to the boundary of
$\Lambda \setminus W(N) $.

{\bf iii)} The manifold  $K$ is a fibre bundle over $N$ with fiber  $(S^{r-1} \times
S^{m-n-r-1})$.

{\bf iv)} The leaves of $\fa$ near  $\Lambda$ which are contained in
different connected components of $M\setminus \Lambda$ are obtained
from each other by performing Milnor-Wallace surgery fiberwise. More
precisely, up to homeomorphism, to get the leaves in one side we
attach to  $\Lambda \setminus W(N) $ the obvious bundle with fiber  $(D^{r} \times
S^{m-n-r-1})$, and to get the leaves in the other side we attach the corresponding
bundle with fiber  $(S^{r-1} \times D^{m-n-r})$. To recover $\Lambda$
we just collapse to a point each  $(S^{r-1} \times
S^{m-n-r-1})$-fiber in $K$.

\end{Theorem}

In the sections below we give   examples of isolated saddles on
3-manifolds with Morse index 1 having only one global  separatrix
and also examples having two separatrices.

\section{Pairings and Foliated surgery}
\label{Section:foliatedsurgery}

%\subsection{Pairings and Surgery for Bott-Morse foliations}\label{subsection:DeadBrachesdimensiontwo}

Let us motivate what follows by recalling the classical
center-saddle bifurcation in the plane, which is depicted in
Figure~\ref{Figure:centroseladim2trivial0} in page \pageref{Figure:centroseladim2trivial0}.  We have an isolated center in the plane and a saddle in
the boundary of its basin. These are replaced via an isotopy  by a
trivial foliation. This elimination procedure may be described as
follows. Consider the 1-parameter family of vector fields $Z_\ve = (x_1^2-\ve)\frac{\po}{\po x_1} +
x_2\,\frac{\po}{\po x_2}$\,, $\ve > 0$. For $\ve > 0$, this vector field has a
pair  of singularities saddle-source: it has  a saddle-node singularity for
$\ve = 0$ and  no singularity for $\ve < 0$.   Thus
the original pairing center-saddle can be viewed as a deformation
of a trivial vertical foliation via passing through a saddle-node
singularity.

%\begin{figure}[ht]
%\begin{center}
%%\include  graphics[scale=0.4]{center-saddlepicture.eps}
%\caption{Basic picture in dimension two}
%\label{Figure:centroseladim2trivial}
%\end{center}
%\end{figure}

Other center-saddle arrangements in dimension two are depicted in
Figure~\ref{Figure:nontrivial1}. The figure on the left shows a
disc with a  center-singularity which is replaced (or replaces,
depending on the orientation of the arrow) a disc with three
singularities: two of them are centers and one is a saddle which
is in the boundary of the basin of both centers.  The figure on the
right also shows a saddle in the boundary of the basins of two
centers. Notice that there is a significant difference in these
two cases. In the first case the saddle has only one separatrix,
while in the second case it has two separatrices.

 \begin{figure}[ht]
\begin{center}
\includegraphics[scale=0.4]{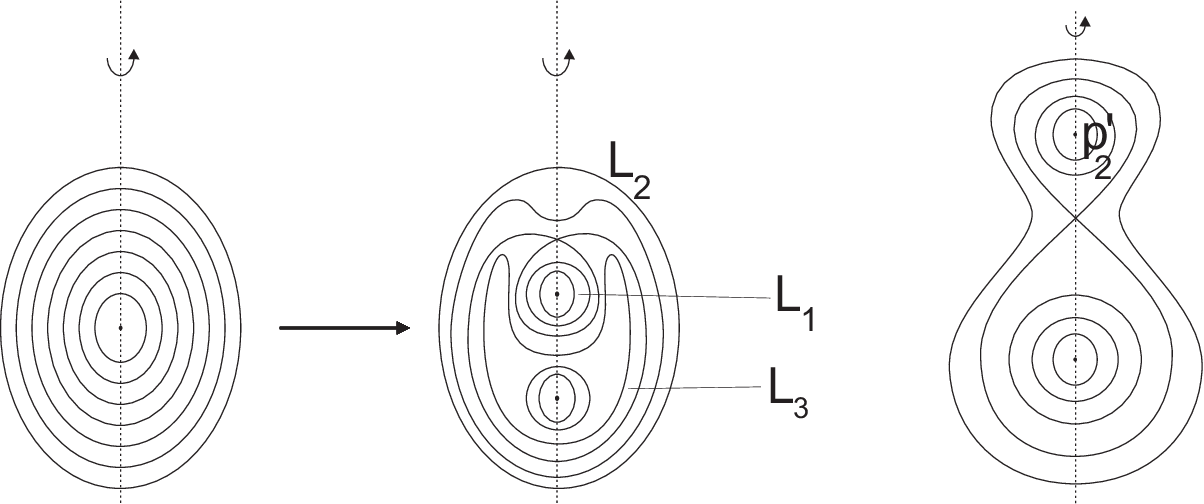}
\caption{ }
\label{Figure:nontrivial1}
\end{center}
\end{figure}

More generally,

\begin{Definition}[Pairings]
\label{Definition:pairing} {\rm  Let $\fa$ be a foliation with Bott-Morse singularities on $M$.
 We say that a saddle component
$N_0\subset \sing(\fa)$ and a center component $N_1\subset
\sing(\fa)$  are {\it paired},  denoted $N_0 \leftrightarrow
N_1$, if $N_0$ is in the boundary of the basin of $N_1$, {\it
i.e.},  $N_0\subset \partial \mathcal C(N_1,\fa)$.  If $N_0$ is paired with several center components $N_1,...,N_r$ we denote it
by $N_0 \leftrightarrow (N_1,...,N_r)$.}
\end{Definition}

\begin{Definition}[Foliated Surgery]
\label{Definition:surgery}
  { \rm Let $\fa$ be a Bott-Morse
foliation on $M^m$, and let $N_0$ be a saddle component which is paired with some center components $N_1,...,N_r$, $r \ge 1$.
 A {\it foliated surgery} for the pairing $N_0 \leftrightarrow
 (N_1,...,N_r)$
 means a choice of a compact, connected region $R \subset M$ containing
$N_0, N_1,...,N_r$ in its interior $\buildrel {\circ} \over {R
\;}$, such that:

\begin{enumerate}

\item The restriction of the foliation $\fa$ to $R$ can be
replaced by a Bott-Morse foliation $\fa_m$ on $R$ which coincides
with $\fa$ in a neighborhood of the boundary of $R$;

\item The new foliation on $M$, also denoted $\fa_m$, is
Bott-Morse;

\item If $\fa$ is proper, so is $\fa_m$.

\end{enumerate}}

\end{Definition}

Foliated surgery can be used   to reduce the singularities of
Bott-Morse foliations. The following definitions make this idea
precise.

\begin{Definition}[Reducible and trivial pairings]
\label{Definition:Redpairing} Let {\rm
 $\fa$ be  a  closed Bott-Morse foliation on a connected
closed  oriented $m$-manifold $M$.

\begin{enumerate}

 \item We say that a pairing $N_0 \leftrightarrow
(N_1,...,N_s)$ is {\it trivial}  if there is a foliated surgery
for this pairing  in a region $R $,
  such that   the resulting foliation is regular on $R$.
 In this case we say that the new
foliation $ \fa_m$ is obtained from $\fa$ by {\it eliminating} the
components $N_0, N_1,...,N_s$.

\item   We say that a pairing $N_0 \leftrightarrow (N_1,...,N_s)$
is {\it reducible} if there is a foliated surgery for it in a
region $R $, such that the singular set of the resulting foliation
$ \fa_m$ has at most $s$
 components in $R$. In this  case we say that
  $ \fa_m$ is obtained from $\fa$ by {\it replacement} or
{\it reduction} of the components $N_0, N_1,...,N_s$.

  \item A pairing which is not reducible is called {\it
irreducible}.
\end{enumerate}}
\end{Definition}

\begin{Remark}
{\rm We will see later
 that the type of
surgeries we perform for proving Theorem~\ref{Theorem:Center-Saddledimensionthree} also satisfy the following
condition:

\vglue.1in \par{4.} The holonomy pseudogroup of $\fa_m$ is
obtained from the holonomy pseudogroup of $\fa$ {\it by reduction}
in the following sense:} \end{Remark}

\begin{Definition}{\rm
In the above situation, we say that the holonomy pseudogroup of $\fa_m$ is {\it obtained
by reduction} from the holonomy pseudogroup of $\fa$ if there is a
collection $\mathcal T=\{T_j\}_{j\in J}$ of sections $T_j \subset
M$ transverse to $\fa_m$ such that:

i)  Each leaf of $\fa_m$ intersects some $T_j$;

ii) Each $T_j$ is also transverse to $\fa$, though it may happen
that not all leaves of $\fa$ intersect the union $\bigcup
\limits_{j\in J} T_j$;

iii) The holonomy pseudogroup $\Hol(\fa_m,\mathcal T)$ of $\fa_m$
with respect to $\mathcal T$  is isomorphic to the subpseudogroup  of the
holonomy pseudogroup  of $\fa$ generated by the holonomy maps
 corresponding to paths $\gamma$
contained each in some leaf $L$ of $\fa$ and joining a point in
$T_i\cap L$ to a point in $T_j \cap L$.}
\end{Definition}

\medskip

\subsection{Pairings in dimension three}

We now introduce certain types of possible center-saddle pairings
that appear in dimension three for closed Bott-Morse foliations
$\fa$. These give all possible pairings satisfying the conditions
of Theorem~\ref{Theorem:Center-Saddledimensionthree}.

\subsubsection{Bundle-type pairings}

\noindent We have two different bundle-type pairings obtained by ``lifting"  two-dimensional pictures to a product $S^1
\times D^2$, where $D^2$ is homeomorphic to a $2$-disc. We
recall that every oriented 2-disc bundle over $S^1$ is trivial and, moreover,
every circle embedded in an oriented manifold has trivial normal bundle.

Just as in \ref{Rem:notptoducts}, the foliations we give on $S^1
\times D^2$ are locally products of a 2-dimensional picture
by an interval, but they may not be products globally, though for
simplicity in the pictures below we always represent the product
 foliations. Pairing (P.NI.1) (abbreviation for ``product-nonisolated" of type $1$) is depicted in
Figure~\ref{Figure:dimensiontwo0} and consists  of an $S^1$-saddle
component and an $S^1$-center component. This pairing is trivial.

\begin{figure}[ht]
\begin{center}
\includegraphics[scale=0.8]{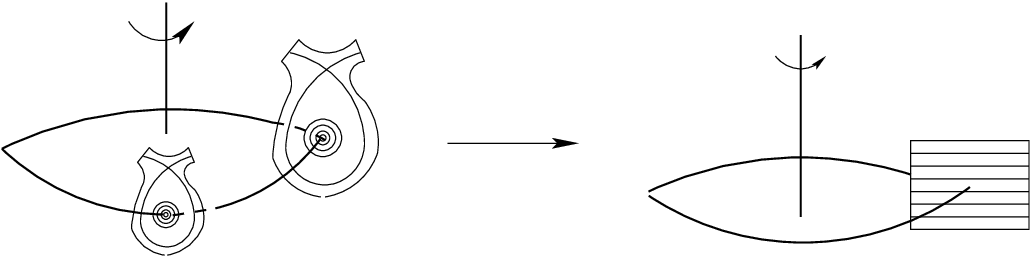}
\end{center}
\caption{P.NI.1}
 \label{Figure:dimensiontwo0}
\end{figure}

The second bundle-type  pairing is (P.NI.2) (abbreviation for ``product-nonisolated" of type $2$), obtained as the lift to $S^1
\times D^2$ of  a pairing as in Figure~\ref{Figure:dimensiontwo}
below. Here an $S^1$-saddle
 component is paired with two $S^1$-center components. One of these is
 the previous trivial
 pairing, the other is the non-trivial pairing (P.NI.2).

\begin{figure}[ht]
\begin{center}
\includegraphics[scale=0.45]{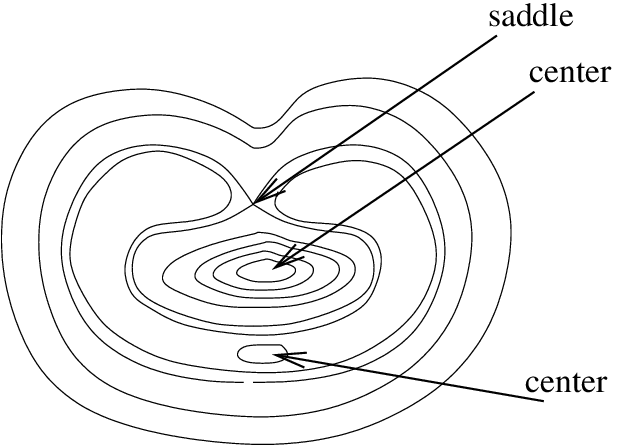}
\end{center}
\caption{P.NI.2} \label{Figure:dimensiontwo}
\end{figure}

\subsubsection{Isolated singularities}
The list of possible pairings of isolated singularities in
dimension $3$ is given in \cite{Camacho-Scardua}.
 The first is a trivial pairing (T.I) (for ``trivial-isolated")
  as in Example~\ref{Example:trivialfoliation}.
 This is obtained as rotation of the two-dimensional picture in
 Figure~\ref{Figure:centroseladim2trivial0}
 with respect to the vertical axis. The second, (NT.I.1)
 (for nontrivial-isolated of type
 $1$) is depicted on the left in
 Figure~\ref{Figure:isolated1}; it is a non-trivial pairing with spherical leaves inside the
 region bounded by the separatrices and tori outside.
  On the right in  Figure~\ref{Figure:isolated1} we have a non-trivial, isolated
   pairing of type $2$ (NT.I.2),
 obtained by rotation of the image with respect to the {\it
horizontal} axis. The pairing consists of an isolated saddle at
the origin and a center {\em at infinity}. The outer and the inner
leaves are $2$-spheres and the picture can be completed by two
isolated centers (on the left and on the right of the picture)
each one of them defining  a trivial pairing of type $(T.I)$ with
the saddle.

 \begin{figure}[ht]
\begin{center}
\includegraphics[scale=0.4]{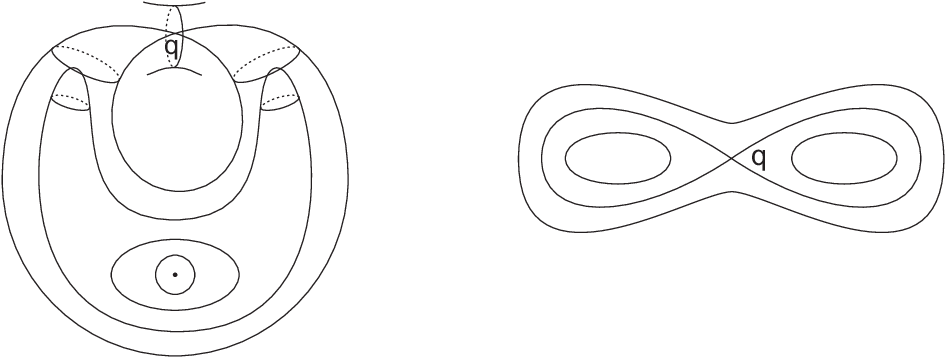}
\caption{Pairings NT.I.1 and NT.I.2}

\label{Figure:isolated1}
\end{center}
\end{figure}
\vglue.1in

Another rotation type pairing is  depicted in
Figure~\ref{Figure:isolated2}. Here we start with   an isolated
center $C_0$ at the origin in a two dimensional picture. Then we
introduce a  (two-dimensional)  center saddle pairing of type
trivial isolated, and we rotate it with respect to the vertical
axis. The pairing we consider is the one given by the original
center $C_0$ and the saddle. This pairing is non trivial and will
be called (NT.I.3) (for ``nontrivial-isolated" of type $3$).

 \begin{figure}[ht]
\begin{center}
\includegraphics[scale=0.4]{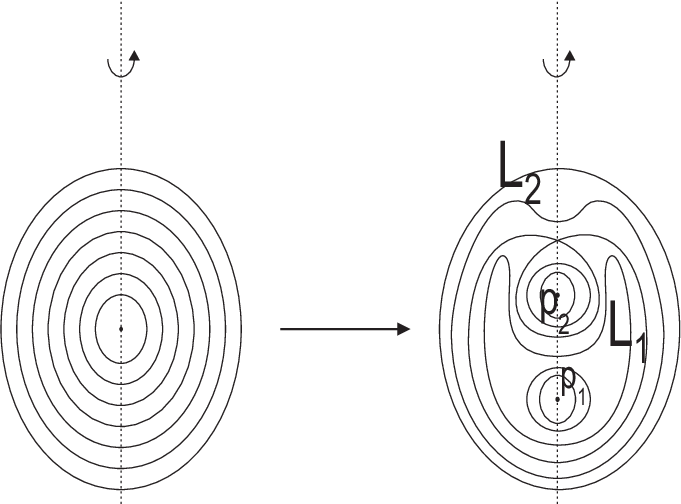}
\caption{Pairing $(NT.I.3)$} \label{Figure:isolated2}
\end{center}
\end{figure}

%\newpage

Notice that the 2-dimensional picture of this pairing is like in
Figure~\ref{Figure:dimensiontwo}. The difference is that for that
pairing we take the product with $S^1$, while here we are making
the 2-dimensional model rotate with respect to an axes as
indicated in Figure~\ref{Figure:isolated2}.

\subsubsection{Non-isolated saddle}
Rotation of  Figure~\ref{Figure:notproduct2} with respect to the
vertical axis gives two isolated centers paired with a
non-isolated saddle. These will be called (U.NI.) (for
 ``upper-nonisolated") and the other center-saddle pairing will be
called (L.NI.) (for ``lower-nonisolated").

\begin{figure}[ht]
\begin{center}
\includegraphics[scale=0.4]{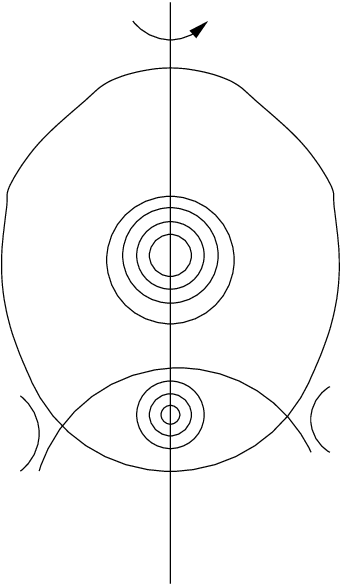}
\caption{Pairings $(U.NI.)$ and $(L.NI.)$. }
\label{Figure:notproduct2}
\end{center}
\end{figure}

 One also has a
product type pairing    (P.NI.) (for
`product-nonisolated") obtained by rotating the
Figure~\ref{Figure:sphere3} with respect to the vertical axis. The
pairing we consider consists of the non-isolated center and the
non-isolated saddle.

\begin{figure}[ht]
\begin{center}
\includegraphics[scale=0.5]{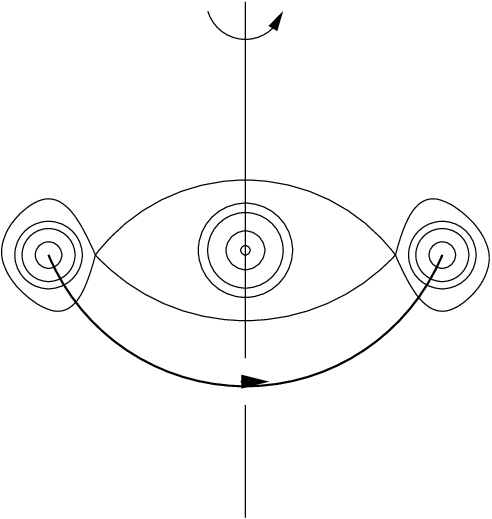}
\caption{Pairing $(P.NI.)$} \label{Figure:sphere3}
\end{center}
\end{figure}

\subsection{Irreducible components}~\label{subsection:irreduciblepairingsdim3}
We describe the irreducible  pairings in dimension three that
appear in the situation we study, {\it i.e},. $\fa$ closed Bott-Morse foliation with $c(\fa)> 2
s(\fa)$:

\vglue.1in

\noindent{\bf 1 - Disconnected irreducible component of bitorus
type}. We consider an isolated saddle $N_0$ and two non-isolated
centers $N_1, N_2$ as singular set of a foliation in a solid torus
as described below. The union $\Lambda(N_0)$ of the saddle and its
separatrices is homeomorphic to a bitorus obtained by gluing
 three pieces $P_1, P_2, C$ where each $P_j$ is a torus minus a disc,
$C$ is a cylinder,   and then collapsing one of the boundary curves
of $C$ into a point $N_0$. The outside leaves are diffeomorphic to
the bitorus (see figure \ref{Figure:bitorustype}). The complement
$\Lambda(N_0)\setminus \{N_0\}$ has two connected components
$\Lambda_1, \Lambda_2$ such that $\Lambda_j \cup \{N_0\}$ is
homeomorphic to a torus pinched at the point and bounds a region
$R_j$ foliated by tori. We have a center $N_j\subset R_j$ such
that $\C(N_j,\fa)=R_j$ and $\partial \C(N_j,\fa)=\Lambda_j$. We
call this case {\it disconnected irreducible} case, because
 $\Lambda(N_0) \setminus N_0=\Lambda_1 \uplus \Lambda_2$
is the disjoint union of two components. Notice this is a special
case of Example \ref{multiple irreducible}.

\begin{figure}[ht]
\begin{center}
\includegraphics[scale=0.45]{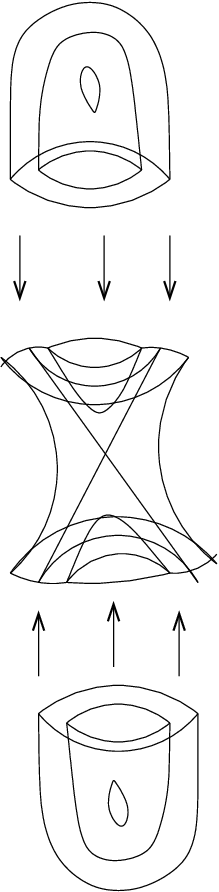}
\end{center}
 \caption{Irreducible component of bitorus type. }
 \label{Figure:bitorustype}
\end{figure}

 \vglue.1in

\noindent{\bf 2 - Torus-ball type component}. A second model is
obtained with a torus $T$ bounding a solid torus $\Omega$ in
$\mathbb R^3$. Inside $\Omega$ we foliate a neighborhood of
$\partial \Omega=T$ by concentric tori and introduce an isolated
saddle singularity $N_0$ by collapsing the boundary of a disc into
a point $N_0$ so that the resulting variety $\Lambda(N_0)$ is the
union of a $2$-sphere and a torus, both pinched at $N_0$ (see
figure~\ref{Figure:torusballtype}). Now we foliate the region
interior to the sphere by spheres and the region interior to the
torus by tori accumulating to a circle $N_1$. The exterior of
$\Omega$ can be foliated  by tori centered at a circle around a point
 $N_2$ at infinity. Notice that one of the pairings is
trivial of type (T.I).
This gives our second model of
`disconnected singularity".

\begin{figure}[ht]
\begin{center}
\includegraphics[scale=0.35]{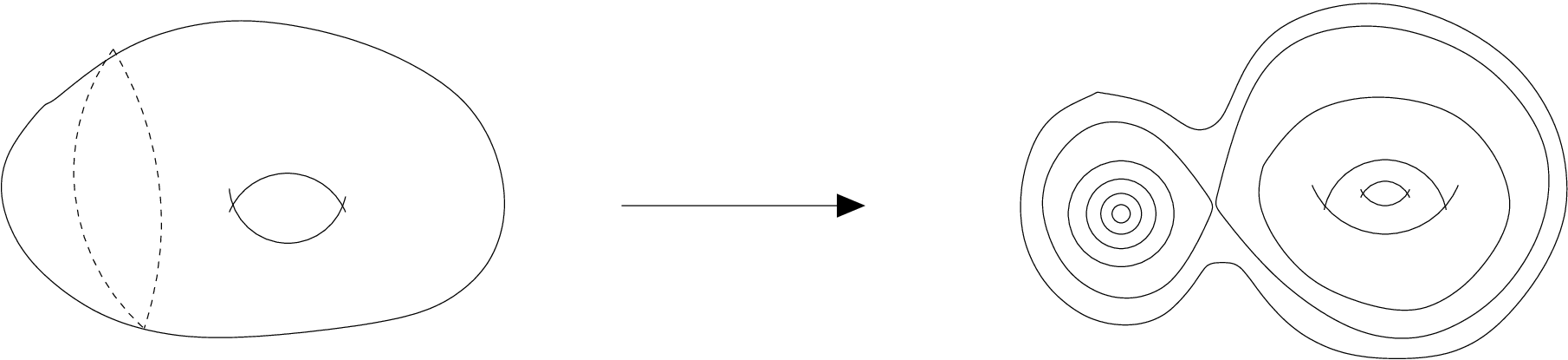}
\end{center}
 \caption{Disconnected component of torus ball type. }
 \label{Figure:torusballtype}
\end{figure}

\begin{Definition}
\label{Definition:irreducicle} {\rm We call the two pairings
constructed above {\it disconnected} pairings. (The name
indicates the fact that the closure of the union of separatrices
is disconnected if we remove
 the saddle.)}
\end{Definition}

\noindent{\bf 3 - Connected irreducible component (only one
center)}. We consider an isolated saddle $N_0$ and a non-isolated
center $N_1$ as singular set of a foliation in a solid  bitorus as
depicted below (see figure~\ref{Figure:onlyonecenter}).  We begin
with a solid torus region $\Omega\subset \mathbb R^3$ with
boundary a bitorus $\partial\Omega$ . Foliate a neighborhood of the
boundary by bitori and fix one of these bitorus;
take a meridian in one of the handles  and collapse it to a point
$N_0$. Call this singular leaf $\Lambda(N_0)$. Foliate now the
interior of the solid region bounded by $\Lambda$ by tori, in such
a way that we have a single component of the singular set, which
is a one-dimensional center $N_1$. The leaves outside $\Lambda$
are diffeomorphic to the bitorus and the inner leaves are tori
which accumulate to  $N_1$. We call this a {\it connected
irreducible} case. Notice that  $\Lambda(N_0)$ is the union of $N_0$
with all the separatrices accumulating to $N_0$ and
$\Lambda(N_0)\setminus N_0$ is connected.

\begin{figure}[ht]
\begin{center}
\includegraphics[scale=0.25]{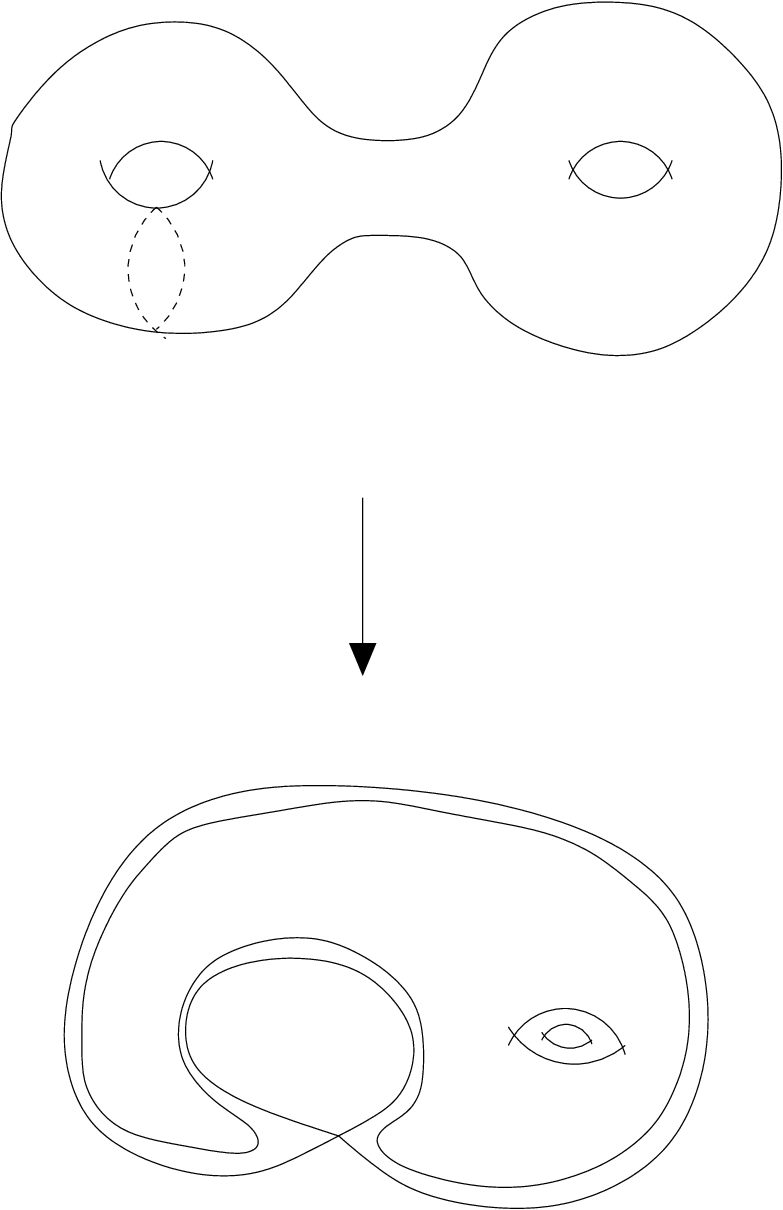}
\end{center}
 \caption{Connected irreducible component with only
 one center and original genus $g=2$. }
 \label{Figure:onlyonecenter}
\end{figure}

\vglue.1in

 \noindent{\bf 4 - Multiple irreducible connected component}. Let
$L(g)$ be an unknotted solid handle-body in $\mathbb R^3$ with
genus $g>1$,  and  foliate $L(g)$ as in Example \ref{multiple
irreducible}. We first have leaves of genus $g$, parallel to the
boundary.  Then a surface $S $ with $g-1$ saddle-points which
bounds in $g$-components, each diffeomorphic to an open solid
torus $T_j, \, j=1,...,g$. Then we foliate each $T_j$ in the usual
way, by copies of $S^1 \times S^1$, having in each torus a circle
$N_j$ as singular set, all of center-type.

%\begin{figure}[ht]
%\begin{center}
%&%\include  graphics[scale=0.45]{figBM21.eps}
%\end{center}
%\caption{Non-isolated saddle and centers pairings, only one
%trivial.} \label{Figure:product1}
%\end{figure}

%\begin{figure}[ht]
%\begin{center}
%&%\include   graphics[scale=0.45]{figBM23.eps}
%\end{center}
%\caption{Non-isolated saddle and centers trivial pairings.}
%\label{Figure:product2}
%\end{figure}

\section{Bott-Morse foliations on $3$-manifolds}
\label{section:BottMorsedimension3}
We consider a closed Bott-Morse foliation $\fa$ on a 3-manifold $M$. Let
  $N_1, N_2$ be distinct center-type components of $\sing(\fa)$ such that
$\partial \mathcal C(N_1, \fa)\cap
\partial\mathcal C(N_2, \fa)\ne \emptyset$. We thus know from
Theorem~\ref{Lemma:basinopen} that  there is
exactly one saddle component $N_0$ in $\partial \mathcal C(N_1,
\fa)\cap
\partial\mathcal C(N_2, \fa)$, and
 Theorem~\ref{topology of boundary} says that each $\partial \mathcal
C(N_j,\fa)$ is the union of   $N_0$ and   separatrices of $N_0$. This implies:

\begin{Lemma} $\,$
\label{Lemma:existssaddlebdbasin} We can have one of the following
 possibilities:

{\bf i)}$\;$ The saddle $N_0$ has only one separatrix. In this
case one has $\partial \mathcal C(N_1, \fa) =
\partial\mathcal C(N_2, \fa)$.

{\bf ii)}  The saddle $N_0$ has two separatrices: then $\partial
\C(N_1,\fa)\cap
\partial \C(N_2,\fa)$ is a union of $N_0$ and separatrices of
$N_0$, and one can have:
\begin{itemize}

\item    $\partial \C(N_1,\fa)\cap
\partial \C(N_2,\fa)=N_0$.

\item   $\partial\mathcal C(N_1, \fa)\subsetneqq
\partial\mathcal C(N_2, \fa) \,$ or viceversa,
$\partial\mathcal C(N_2, \fa)\subsetneqq
\partial\mathcal C(N_1, \fa)\,,$ and $\partial\mathcal C(N_1,
\fa)\cap \partial\mathcal C(N_2, \fa) \,$ is $N_0$ union a
separatrix.

\item  $\partial\mathcal C(N_1, \fa) =
\partial\mathcal C(N_2, \fa) $.
\end{itemize}

\end{Lemma}

%%%%%%%%%%%%%%%%%%%%%%%%%%%%%
\medskip

 We now describe the possible arrangements for $N_0, N_1,
N_2$ and the basins $\mathcal C(N_1,\fa), \mathcal C(N_2,\fa)$.
The possible cases are divided according to the following
hierarchy:
\begin{itemize}
\item[{\rm(1)}] the dimension of the saddle $N_0$,

\item[{\rm(2)}] the dimensions of the centers $N_1, N_2$, and

\item[{\rm(3)}] whether or not $\partial \C(N_1,\fa)\cap
\partial \C(N_2,\fa)$ is contained in $\sing(\fa)$, {\it i.e.}, whether or not \\
$\partial \C(N_1,\fa)\cap
\partial \C(N_2,\fa)= N_0$.
\end{itemize}

\subsection{Isolated saddle}
We assume as above that   $\fa$ is a closed Bott-Morse foliation on a closed 3-manifold $M$.
  We begin with the
case of
isolated singularities studied in \cite{Camacho-Scardua}.    The
following two results are respectively  Lemma~3 and Lemma~2 in
 \cite{Camacho-Scardua}:

\begin{Lemma}
\label{Lemma:eliminationallisolated}  If an isolated saddle
singularity $q$ is such that $q\in\partial \mathcal C(p_j,\fa)$
for two isolated centers $p_1, p_2$, then one of the pairings $q
\leftrightarrow p_j$ is trivial.
\end{Lemma}

\begin{Lemma}
 \label{Lemma:separatrices}  If  $p
\in \Sing(\fa)$ is an isolated  center and $q$ is an isolated saddle
contained in $\partial \mathcal C(p,\fa)$, then
 we have the following possibilities:

\begin{itemize}
\item[{\rm(i)}] If $\po \mathcal C(p,\fa)\setminus \{q\}$ is
connected, then:
\begin{itemize}
\item [{\rm (a)}] either $\po \mathcal C(p,\fa)$ is homeomorphic
to a sphere $S^2$ with a pinch at $q$ {\rm(}a tear drop{\rm)} and
the pairing
 $q \leftrightarrow p$  is trivial; or

\item[{\rm (b)}] $\po \mathcal C(p,\fa)$ is homeomorphic to a
pinched torus, obtained from a torus $S^1 \times S^1$ by
collapsing a meridian to a point.
\end{itemize}

\item[{\rm(ii)}] If $\po\mathcal C(p,\fa)\setminus \{q\}$ has two
connected components, then  $\po \mathcal C(p,\fa)$ is the union
of two spheres $S^2$ which meet at $q$.
\end{itemize}
\end{Lemma}

An example of type (i.a) is  the basin of the center $p_2$ in
Figure~\ref{Figure:isolated2}, while the basin of $p_1$ in that
same picture illustrates the case  (i.b). An example of type (ii)
is given in figure 12, corresponding to the pairing NT.I.1.

%As a complement to the above we have:

%\begin{Lemma}
%\label{Lemma:complementtocamacho} Let  $\fa$ be as in
%Lemma~\ref{Lemma:separatrices}.  If \, $\partial \mathcal
%C(p,\fa)\cap \sing(\fa)$ has dimension one, then $\partial
%\mathcal C(p,\fa)\cap \sing(\fa)=N_0$ is a saddle type component
%and the boundary $\partial \mathcal C(p,\fa)$ is homeomorphic to a
%$2$-sphere. Indeed, it is a 2-sphere obtained as union of two
%$2$-discs glued along their common boundary where the sphere is
%singular.
%\end{Lemma}
%\begin{proof}

%\end{proof}

\medskip

Next we have the case of   an isolated saddle paired
with two non-isolated centers.

\begin{Lemma}[$N_0=\{q\}$, two non-isolated centers.]
~\label{Lemma:someisnotisolated}~\label{Lemma:someisnotisolatednonsingular}
Let $N_1, N_2$ be one-dimensional center-type components and $N_0$
an isolated saddle such that $N_0 \subset \partial \C(N_1,\fa)
\cap
\partial \C(N_2,\fa)$.
Assume further that $\fa$ satisfies the inequality $c(\fa) > 2
s(\fa)$. Then $N_0, N_1, N_2$ are in a same disconnected
 component of $\fa$ (as in
Definition~\ref{Definition:irreducicle}).  This is of bitorus type if $N_0 =
\partial \C(N_1,\fa) \cap \partial \C(N_2,\fa)$, or of torus-ball
type if otherwise.
\end{Lemma}

\begin{proof}
  Let $U = U(N_0)$ be a distinguished neighborhood of
$N_0$. By Corollary~\ref{Lemma:basinopen-dim3} each leaf $L\subset
 \mathcal \C(N_j,\fa), \, j=1,2$, is a torus and the interior of
 the basin $\mathcal C(N_j,\fa)$ is a solid torus. The local separatrices of
the saddle $N_0$ divide the neighborhood $U$ into three open
regions  $\buildrel{\circ}\over{R}_1,
\buildrel{\circ}\over{R}_2, \buildrel{\circ}\over{R}_3$ as in
Figure~\ref{Figure:holonomy}, glued along the separatrices.
    For each $j = 1,2,3$, denote by $R_j$
the topological closure of $\buildrel{\circ}\over{R}_j$ in $U$.

\medskip

Let us now assume that $\partial \C(N_1,\fa)\cap
\partial \C(N_2,\fa)= N_0$.  Then, up to reordering the regions $R_1$ and
$R_2$, we have  $\partial \C(N_j,\fa)\cap U  \subset R_j$ for
$j=1,2$. In this case, given  leaves $L_j\in \C(N_j,\fa)$ for
$j=1,2$ we have that $L_j\setminus (L_j\cap U )$ is a torus minus a
 disc. This shows, by Theorem~\ref{Theorem:partialstabilitytheorem},
 that a leaf $L$ of $\fa$ with $L\cap U \subset R_3$ is the union
 of two tori minus a disc in each, and a cylinder which corresponds to the
 hyperboloid leaf on $R_3$; hence $L$ is a bitorus. This shows the existence
 of an invariant
 solid bitorus $B\subset M$ containing a neighborhood of $N_0\cup N_1 \cup
 N_3$, proving the lemma in this case.

\medskip

Let us now assume that $\partial \C(N_1,\fa)\cap \partial
\C(N_2,\fa)$ contains also some leaf of $\fa$. By
Lemma~\ref{Lemma:existssaddlebdbasin}, $\partial \C(N_1,\fa)\cap
\partial \C(N_2,\fa)$ is union of $N_0$ and separatrices of $N_0$.
Up to reordering the regions $R_1$ and $R_2$, we have two
possibilities: either $\C(N_1,\fa)\cap U  \subset R_1$ and
$\C(N_2,\fa)\cap U  \subset R_3$, or else  $\C(N_1,\fa)\cap U \cap R_j\ne
\emptyset$ for $j=1,2$ and $\C(N_2,\fa)\cap U \subset R_3$.

\medskip

\noindent{\bf Case 1}.  Assume $\C(N_1,\fa)\cap U  \subset R_1$ and
$\C(N_2,\fa)\cap U  \subset R_3$. Then given a leaf
$L_1\in \C(N_1,\fa)$ we have that $L_1\setminus(L_1\cap U )$ is a
torus minus a disc; and given a leaf $L_2\subset \C(N_2,\fa)$ we
have that $L_2 \setminus (L_2 \cap U )$ is a torus minus a cylinder,
neighborhood of a curve that bounds a disc.
Notice that if $\Lambda(N_0)$ is the union of $N_0$ with all the
separatrices accumulating to $N_0$, so $\Lambda(N_0)\setminus
N_0$ may have one or two connected components.

Suppose  first that $\Lambda(N_0)\setminus
N_0$ is
connected. Then, since the regions $R_1$ and $R_3$ are adjacent,
Theorem~\ref{Theorem:partialstabilitytheorem} implies that
$L_1\setminus (L_1\cap U(N_0)) $ and $L_2\setminus (L_2 \cap
U(N_0))$ are homeomorphic what is absurd, so this case cannot
occur. Thence $\Lambda(N_0)\setminus N_0$  has two connected components $ \Lambda_1 \cup
\Lambda_2$,    say  with $\partial
\C(N_1,\fa)\cap \Lambda_1\ne \emptyset$.

We have $\partial
\C(N_1,\fa)=\Lambda_1 \cup N_0$ and
Theorem~\ref{Theorem:partialstabilitytheorem} implies that for a
leaf $L_2\subset \C(N_2,\fa)$ we have that $L_2\setminus U(N_0)$ has
two connected components $L_2 ^1 $ and $L_2 ^2$ with $L_2 ^1$
close to $\Lambda_1$ in the sense of
Proposition~\ref{Proposition:trivialneighborhood}. Therefore
$L_2^1$ is a torus minus a disc and the same holds for
$\Lambda_1\setminus (\Lambda_1\cap U(N_0))$.

By the local
description of $\fa$ in $U(N_0)$ we have that $\Lambda_1$ is
homeomorphic to a torus pinched at a point. By
Proposition~\ref{Proposition:trivialneighborhood},  given
$L_2\subset \C(N_2,\fa)$ we have that  the component $L_2 ^1$ is a torus minus a
disc. By the local form of $\fa$ in $U(N_0)$ we have that $L_2 \cap U(N_0)$ is
a cylinder and, since $L_2$ is a torus and $L_2^1$ is a torus
minus a disc, we have that $L_2^2$  is a disc.

A leaf $L$
contained in the interior of the region bounded by $\Lambda(N_0)$
and close enough to $\Lambda_2$ must be a 2-sphere (it is the
union of two discs, one given by
Theorem~\ref{Theorem:partialstabilitytheorem} and the
homeomorphism with $L_2 ^2$, and the other given by the local type
of the leaves of $\fa$ in the region $R_2$), proving the lemma in this case.

\medskip
\noindent{\bf Case 2}.  We have $\C(N_1,\fa)\cap U \cap R_j\ne
\emptyset$ for $j=1,2$ and $\C(N_2,\fa)\cap U \subset R_3$. It is
then clear that all the separatrices of $N_0$ are contained in
 $\partial \C(N_1,\fa)\cap \partial\C(N_2,\fa)$. Since there are no
 saddle connections we have that $\partial
\C(N_1,\fa)\cap \partial\C(N_2,\fa)=\Lambda(N_0)$ is the union of
$N_0$ and all the separatrices of $N_0$.  And because $\C(N_1,\fa)\cap
R_j\ne \emptyset$ for $j=1,2$ we have that $\Lambda(N_0)\setminus
N_0$ is connected.

Notice that given a leaf $L_1\in \C(N_1,\fa)$ we have
that $L_1\setminus(L_1\cap U )$ is connected and equal to a torus
minus two discs.  Thus $L_1\setminus (L_1\cap U )$ is a cylinder
with a single handle. Given a leaf $L_2\subset \C(N_2,\fa)$ we
have that the intersection  $L_2 \cap U $ is a cylinder.
Theorem~\ref{Theorem:partialstabilitytheorem} implies that
$L_1\setminus (L_1\cap U ) $ and $L_2\setminus (L_2 \cap U )$ are
homeomorphic to $\Lambda(N_0)\setminus (\Lambda(N_0)\cap U )$.
This and the local description of $\fa$ in $U $ show that the
leaves $L_2\subset \C(N_2,\fa)$ are  homeomorphic to the union of
a torus minus two discs with a cylinder $S^1 \times [0,1]$, so
$L_2$ is a bitorus. On the other hand
Corollary~\ref{Lemma:basinopen-dim3} implies that the leaves
$L_2\subset \partial \C(N_2,\fa)$ must be tori, hence this case
cannot occur.
  \end{proof}

\begin{Remark}
\label{Remark:irreduciblecase} {\rm Consider the above setting
with the irreducible component  being of bitorus type. Let $B$ be
the invariant bitorus  in that proof and let $\Omega\subset M$
 be defined as the union of $B$
  and all leaves $L$ homeomorphic to the
 bitorus and such that $L$ bounds a region $R(L)$ containing
 $N_0\cup N_1 \cup N_2$. By the above arguments, $\Omega$
 contains a neighborhood of $N_0\cup N_1\cup N_2$ and all the
 separatrices of $N_0$. We claim that its closure is $\ov\Omega=\Omega\cup
 \{N_0^\prime\}$ for another saddle $N_0^\prime$. Indeed, by the
 local triviality of $\fa$, in a neighborhood of a compact leaf
 $\Omega$ is open in $M\setminus \sing(\fa)$. By  Corollary~\ref{Lemma:basinopen-dim3}
 every leaf in a neighborhood of a center component is a torus or
 a sphere, therefore $\partial \Omega$ contains no center. Thus
 $\partial \Omega$  is  either empty or it contains a saddle. If $\partial \Omega=\emptyset$ then $M=\Omega$
and we have that $\sing(\fa)$ consists of two centers and one saddle,
contradicting the hypothesis $c(\fa) > 2 s(\fa)$. Therefore
we must have another saddle  in $
\partial \Omega$, and  $s(\fa) \geq 2$. This remark will be
 used in the proof of Theorem~\ref{Theorem:onlywhatweneed}.}
\end{Remark}

\begin{Lemma}[$N_0=\{q\}$, centers of mixed dimensions]
\label{Lemma:nottwononisolated} Let $N_0$ be  an isolated saddle
paired with $N_1, N_2$, where $N_1$ is isolated and $N_2$ is
non-isolated. Then either $N_0 \leftrightarrow N_1$ is a trivial
pairing or  it is a non-trivial pairing as in
Figure~\ref{Figure:isolated1} {\rm(}left picture{\rm)}. In the
second case we can choose an invariant region $R$ containing $N_0$
and $N_1$, diffeomorphic to a solid torus, where we can replace
$\fa\big|_R$ by a trivial foliation by  tori around  a
non-isolated center.
\end{Lemma}

\begin{proof}
By hypothesis $N_0$ and $N_1$ are isolated singularities and we
assume  the pairing   $N_0 \leftrightarrow N_1$ is non-trivial. By
Lemma~\ref{Lemma:separatrices} we have a picture as in
Figure~\ref{Figure:isolated1}. We claim that the leaves $L$ close
enough to $\partial \mathcal C(N_1,\fa)$, but not contained in
$\mathcal C(N_1,\fa)$, are diffeomorphic to tori. Indeed such a
leaf is a compact orientable surface and there two possibilities: either
it is a torus obtained as the union of two cylinders, one
``bigger" given by the triviality of the holonomy of $\partial
\mathcal C(N_1,\fa)$ and other ``smaller" given by the local
structure of $\fa$ around the (isolated) saddle $N_0$, or it is
the union of a ``big" cylinder with two discs, resulting in a
$2$-sphere. In this last case all leaves near
$N_0$, except for   the separatrices, are spheres and this is not
possible by a standard homology argument (see Lemma 3 in
\cite{Camacho-Scardua} for details).  We can therefore replace $\fa$ in $V$ by a foliation with a
non-isolated  center as singular set. This proves the lemma. \end{proof}

Notice that a leaf $L$ as in the proof above
necessarily bounds two  solid tori invariant by $\fa$: one is
$R(L)\subset \mathcal C(N_2,\fa)$, and the other is union of (the
singular solid torus) $\overline{\mathcal C(N_1,\fa)}$ with a
small neighborhood of $\partial \mathcal C(N_1,\fa)$ bounded by
$L$. Hence the union $\ov{\mathcal C(N_1,\fa)}\cup \ov {\mathcal
C(N_2,\fa)}$ must be all of $M$, and therefore $M$  is
diffeomorphic to a Lens space: these are the only oriented
3-manifolds that are a union of two solid tori, glued along their
common boundary (see \cite[pages 20--21]{Hempel}).

\medskip

Summarizing the above discussion we obtain:

\begin{Proposition}
\label{Proposition:isolatedsaddlecases} Let $\fa$ be a closed
Bott-Morse foliation on a closed $3$-manifold $M$ satisfying
either $c(\fa) > 2 s(\fa)$ or $\sing(\fa)$ has pure dimension and
$c(\fa)> s(\fa)$. Suppose there is an isolated saddle singularity
$N_0\subset\sing(\fa)$ paired with two centers $N_1,
N_2\subset \sing(\fa)$. Then either $N_0, N_1, N_2$ belong to a
same disconnected  component of $\fa$ or one of the centers, say
$N_1$, must be isolated and we have:

\begin{itemize}

\item[{\rm(i)}] If $N_2$ is also isolated then one of the pairings
$N_0\leftrightarrow N_j$ is trivial.

\item[{\rm(ii)}] If $N_2$ is non-isolated then there are two
possibilities: either the pairing $N_0\leftrightarrow N_1$ is
trivial or there is a compact region containing $N_0\cup N_1$,
diffeomorphic to a solid torus, where $\fa$ can  be replaced by a
compact foliation in $S^1 \times \ov D^2$ with a one-dimensional
center-type component as singular set and invariant boundary $S^1
\times S^1$.
\end{itemize}
\end{Proposition}

Notice that in case (i) we can perform a foliated surgery to
eliminate an isolated center and an isolated saddle, while in case
(ii), either we can eliminate an isolated center and an isolated
saddle, or else we can replace an isolated center and an isolated
saddle  by a non-isolated center. In all cases the condition $c(\fa) > 2 s(\fa)$ is preserved and the number of connected components of the singular set is reduced. If the singularities are all
isolated and we start with  $c(\fa) >  s(\fa)$, then this same condition is preserved by all the above reductions.

%\begin{proof} First we assume  $N_j$ is an
%isolated center for $j=1,2$. Then $N_0\in
%\partial\mathcal C(N_1,\fa)\cap
%\partial \mathcal C(N_2,\fa)$ and the lemma follows from
%Lemma~\ref{Lemma:eliminationallisolated} above. The case  $N_j
%\cong S^1$ for $j=1,2$ is not possible by
%Lemma~\ref{Lemma:someisnotisolated}. Finally, assume
%$N_1=\{p_1\}$ is an isolated center  and $N_2 \cong S^1$. In the
%nontrivial case we have a picture like the one on the left in
%Figure~\ref{Figure:isolated1} (see also
%Figure~\ref{Figure:dimensiontwo}).

%The component $N_2$ is located ``at the infinity" with respect to
%the picture above. The region $\mathcal C(N_2,\fa)$ is the solid
%torus complementary to the picture above, foliated in a standard
%manner by concentric tori having  $N_2$ as common core. Thus we
%can replace the pair $N_0 \leftrightarrow N_1$ and obtain a
%foliated region like $\mathcal C(N_2,\fa)$ above, where we have a
%compact Bott-Morse foliation of $S^1 \times \ov D^2$, by
%concentric tori with singular set of dimension one and of center
%type, and invariant boundary.\end{proof}

\subsection{Non-isolated saddle}

Now we study the possible pairings of a non-isolated saddle $N_0$
at the common boundary of two basins $\mathcal C(N_j,\fa), j=1,
2$   of  a closed Bott-Morse foliation $\fa$ on a closed
$3$-dimensional manifold $M$.    We first have a result about the boundary of the basin of a non-isolated center paired with a non-isolated saddle.

\begin{Lemma}
\label{Lemma:solidtorus}  Assume $\fa$ satisfies the inequality
$c(\fa) >  s(\fa)$. Let $N\subset \sing(\fa)$ be a one-dimensional
center-type component which is paired with some
one-dimensional saddle type component $N_0$. Then
$\partial\mathcal C(N, \fa)$ is homeomorphic to a torus or to the
union of two tori intersecting along a common circle which is a
parallel.
\end{Lemma}

The pairings in Figure~\ref{Figure:dimensiontwo} illustrate both
possibilities in this lemma.

\begin{proof}
 Both $N$ and $N_0$ are non-isolated. Recall that every circle in an oriented manifold has trivial normal bundle. So we may take a distinguished neighborhood
$U $ of $N_0$ in $M$ where $\fa$ is equivalent to the lift to a trivial
bundle over  $N_0\cong S_1$ of  a foliation $\fa_1$ with a Morse
singularity in a neighborhood $V$ of the
origin $0\in \mathbb R^2$. We identify $V$ with a disc transversal
to $N_0$. Notice that $U $ is divided into four conical sectors by
the separatrices of $N_0$.

Consider a sequence of leaves of $\fa$
that has $N_0$ in its closure, and let
 $\mathcal L$ be one of these leaves. The trace of  $\mathcal L$ in
 the transversal $V$ consists
of $i$ arcs, for some $i = 1,...,4$. We claim $i$ must be either
$1$ or $2$. In fact, if $i=4$, then the closure $\overline{
\mathcal C(N, \fa)}$ would be also an open set in $M$, so
$\overline{ \mathcal C(N, \fa)} = M$, contradicting the hypothesis
$c(\fa)>  s(\fa)$. If $i=3$, then given a leaf $L\subset
\C(N,\fa)$ we have that $L\setminus(L\cap U )$ is homeomorphic to a  torus
minus three ``parallel" strips, and it has three connected
components. On the other hand, given any leaf
$L_1\not\subset\C(N,\fa)$ which is sufficiently near $N_0$ we have
that $L_1\cap U $ is homeomorphic to a strip and therefore
$L_1\setminus (L_1\cap U )$ has one  connected component, which is
a contradiction to Theorem~\ref{topology of boundary}. Hence $i$
is $1$ or $2$.

Suppose that $i=1$. A leaf $L\subset \C(N,\fa)$ is
a torus and the intersection $L\cap U $ is a strip in $L$ so that
$L\setminus (L\cap U )$ is a cylinder. Notice that {\it a priori}
 the strip we remove
from the torus is not a neighborhood of a parallel, but a
neighborhood of a curve of type $(1,p)$, up to isotopy.
Nevertheless the complement of such neighborhood is also a
cylinder, by  Remark~\ref{Remark:torusstrip} below. Same
observation applies to $\Lambda$  below.

\vglue.1in By Theorem~\ref{Theorem:partialstabilitytheorem} if we
denote by $\Lambda$ the union   of $N_0$ and the separatrices of
$N_0$ contained in $\partial \C(N,\fa)$,  then $\Lambda\setminus
(L\cap U )$ is a cylinder. By Proposition~\ref{Lemma:productN_0},
$\Lambda\cap U $ is a product of $S^1$ by an open interval.
Therefore $\Lambda$ is the union of two cylinders and
$\Lambda\setminus N_0$ is a smooth manifold. This shows that
$\Lambda$ is homeomorphic to a torus. Since $\fa$ has no saddle
connections we have $\partial \C(N,\fa)=\Lambda$. Finally, if
$i=2$ using Theorem~\ref{Theorem:partialstabilitytheorem} and
Proposition~\ref{Lemma:productN_0} and reasoning as above we
conclude that ${\partial \mathcal C(N, \fa)}$ consists of two tori
that meet at $N_0$ (see Figure~\ref{Figure:nontrivial1}).
\end{proof}

\begin{Remark}\footnote{We are grateful to A.
Verjovsky for  this remark.}
\label{Remark:torusstrip} {\rm Let $\gamma$ be a $(1,p)$ type
closed curve in the torus $T^2 \cong S^1 \times S^1$. Then there
is a homeomorphism of the torus mapping $\gamma$ onto a $(1,0)$
type curve (to see this take any matrix $A$ of determinant that
carries $(1,p)$ into $(1,0)$, and consider the corresponding torus
diffeomorphism). Hence if we consider a
tubular neighborhood of $\gamma$ in the torus, then its complement
is homeomorphic to a cylinder. In particular a closed oriented
surface obtained by gluing  two such complements is
diffeomorphic to the torus.
 }
\end{Remark}

Now we have:

\begin{Lemma}[$N_0\cong S^1$,  isolated centers, non-singular intersection of boundaries]
Suppose that  $N_1$ and $N_2$ are isolated centers and $\partial
\mathcal C(N_1,\fa) \cap
\partial \mathcal C(N_2,\fa)\not \subset \sing(\fa)$.  Then  $\partial
\mathcal C(N_1,\fa) \cap \partial \mathcal C(N_2,\fa)$ is a 2-disc
bounded by $N_0$. Moreover,  $\ov{\mathcal C(N_1, \fa)}\cup
\ov{\mathcal C(N_2, \fa)}$ is a closed 3-ball and the pairing
$N_0\leftrightarrow(N_1,N_2)$ is  trivial.

\end{Lemma}

\begin{proof}
Since $N_1$ and $N_2$ are isolated centres, their basins $\mathcal C(N_j,\fa),
\, j=1,2$, are open balls where the foliation is by concentric
spheres $S^2$. We denote by $\Lambda(N_0)$ the union of $N_0$ with
the separatrices through $N_0$. By
Theorem~\ref{Theorem:localproduct}  the holonomy of $N_0\cong S^1$
is trivial and therefore $\fa$ has a bundle-type structure in a
neighborhood $U \subset M$ of $N_0$.
 The boundary $\partial \mathcal
C(N_j,\fa)$ is a subvariety which is a limit of spheres $S^2$
pinched along $N_0\cong S^1$ with a bundle-type structure in $U $.
Thus $\partial \mathcal C(N_j,\fa)$ is a union of two 2-discs
$D_j^{(1)}$ and $D_j ^{(2)}$ along their common boundary $S^1$.
Since by hypothesis $\partial \mathcal C(N_1,\fa)\cap
\partial \mathcal C(N_2,\fa)$ contains some leaf of $\fa$, not
only $N_0$, it follows from the local picture of $\fa$ in $U $
that, up to reordering, we have  $\partial \mathcal C(N_1,\fa)\cap
\partial \mathcal C(N_2,\fa)=D_1^{(2)}=D_2^{(1)}$ and
the union $\ov{\mathcal C(N_1,\fa)}\cup \ov {\mathcal C(N_2,\fa)}$
is homeomorphic  to the closed three-ball $B^3$.

In fact we can assume that $\partial (\ov{\mathcal C(N_1,\fa)}\cup
\ov {\mathcal C(N_2,\fa)})=D_1^{(1)}\cup D_2^{(2)}\cup N_0$, so we
conclude that there are neighborhoods $W_\nu$ of $\partial
(\ov{\mathcal C(N_1,\fa)}\cup \ov {\mathcal
C(N_2,\fa)})=D_1^{(1)}\cup D_2^{(2)}\cup N_0$ such that $\lim
W_\nu = \partial (\ov{\mathcal C(N_1,\fa)}\cup \ov {\mathcal
C(N_2,\fa)})=D_1^{(1)}\cup D_2^{(2)}\cup N_0$.

If a
non-separatrix leaf $L\in \fa$ intersects $W_\nu$ and $L$ is not
contained in $\partial (\ov{\mathcal C(N_1,\fa)}\cup \ov {\mathcal
C(N_2,\fa)})$,  then $(L\cap W_\nu)$ is either a single disc or
the union of two discs, by
Theorem~\ref{Theorem:partialstabilitytheorem}.
 Moreover,  the intersection $L\cap U $ is
diffeomorphic to a bundle over $S^1$ with fiber the interval
$(-\epsilon, \epsilon)$.  This holds for leaves $L$ nearby
 $D_1^{(1)}$ and leaves $L$ nearby $D_2^{(2)}$, so  we have:

\begin{Claim} We can choose a compact neighborhood  $V \subset M$ of
$\ov{\mathcal C(N_1,\fa)}\cup \ov {\mathcal C(N_2,\fa)}$
diffeomorphic to a solid cylinder $D^2\times [0,1]$, whose boundary
$\po V$ consists of two invariant discs $D_1\cong D^2\times
\{0\}$ and $D_2 =D^2\times \{1\}$ and a transverse open  cylinder
$\Sigma\cong S^1\times (0,1)$. The intersection $\sing(\fa) \cap
V$ consists of $N_0, N_1, N_2$ and no other component of
$\sing(\fa)$.
\end{Claim}
Using this claim  we can replace $\fa\big|_V$ by a trivial
foliation by discs, proving the lemma.
\end{proof}

\begin{Lemma}[$N_0\cong S^1$,  isolated centers, singular intersection of boundaries]
Suppose that  $N_1$ and $N_2$ are isolated and   $\partial
\mathcal C(N_1,\fa) \cap \partial \mathcal C(N_2,\fa)=N_0$. Then
there is an invariant region $R$ containing $N_0, N_1, N_2$,
diffeomorphic to $S^2 \times [0,1]$, and $\fa\big|_R$ can be
replaced by a regular foliation by $2$-spheres. Therefore  we can
eliminate the three components $N_0, N_1,N_2$ at once.

\end{Lemma}

\begin{proof}
The situation is depicted in
Figure~\ref{Figure:twocentersonlysingularity}{\rm} below.

%\end{itemize}

\begin{figure}[ht]
\begin{center}
\includegraphics[scale=0.5]{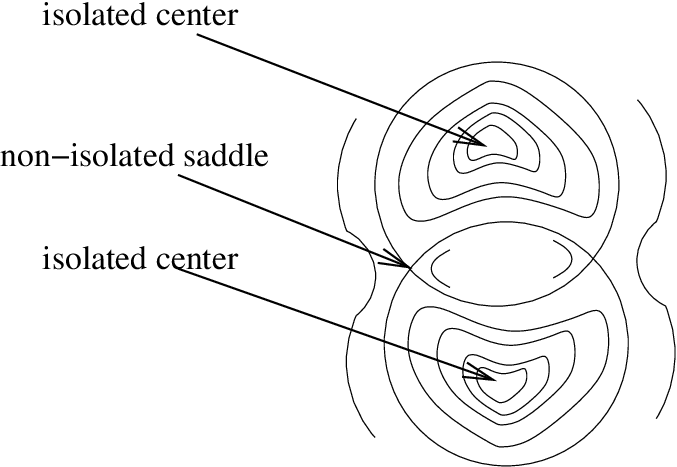}
\end{center}
\caption{} \label{Figure:twocentersonlysingularity}
\end{figure}
We claim that a  leaf $L$ near  $\ov{\mathcal C(N_1,\fa)}
\cup\ov{\mathcal C(N_2,\fa)}$, but not contained in this set, is a $2$-sphere. Indeed, the
intersection of this leaf with a suitable product type
neighborhood $U $, where $\fa$ is of product type, is a strip. On
the other hand, $\partial \mathcal C(N_j,\fa)\setminus \partial
(\mathcal C(N_j,\fa)\cap U )$ is a $2$-sphere minus a strip around
the equator so that $\partial \mathcal C(N_j,\fa)\setminus
\partial (\mathcal C(N_j,\fa)\cap U )$ consists of two disjoint $2$-discs.
Thus, as in the lemmas above,
Theorem~\ref{Theorem:partialstabilitytheorem} implies that
$L\setminus (L\cap U )$ is homeomorphic to the disjoint union of
two $2$-discs. Hence $L$ is the union of two discs $D^2$ and one
strip $[-1,1]\times S^1$. Since it is compact and orientable, we must have that  is
$L$ diffeomorphic to $S^2$.

Therefore we have
a compact invariant neighborhood $R$ of $\ov{\mathcal C(N_1,\fa)}
\cup\ov{\mathcal C(N_2,\fa)}$ diffeomorphic to $S^2\times [0,1]$
and the lemma follows.
\end{proof}

\begin{Lemma}[$N_0\cong S^1$, non-isolated centers, singular intersection of boundaries]
\label{Lemma:eliminationalldimensionone} Assume now that
$N_0\subset \Sad(\fa)$ and $N_1, N_2 \subset \Cent(\fa)$ are all
of dimension one. If $N_0=\partial \mathcal C(N_1,\fa) \cap
\partial\mathcal C(N_2,\fa)$, then there is an invariant solid torus $V$ containing
$N_0, N_1, N_2$. Therefore, we can replace  $\fa\big|_V$ by a
foliation by concentric tori.
\end{Lemma}

\begin{proof}
By hypothesis  $N_0, N_1, N_2$ have dimension one. Given $N_0
\subset \sing(\fa)$ we denote by $\Lambda(N_0)$ the union of $N_0$
with the separatrices through $N_0$. By
Proposition~\ref{Lemma:productN_0}, $\fa$ has a product structure
in a distinguished neighborhood $U \subset M$ of $N_0$
diffeomorphic to a product $D\times S^1$, where $D\subset \mathbb
R^2$ is a product of open intervals. The separatrix $\Lambda(N_0)$
divides $U $ into four regions $R_1, R_2, R_3, R_4$ and we can
assume that $R_j$ is adjacent to $R_{j+1}$ and $R_4$ to $R_1$.

 Considering the  several
possibilities for the trace of $\C(N_1,\fa)$ and $\C(N_2,\fa)$ in
$U $ under the assumption that $\C(N_1,\fa)\cap \C(N_2,\fa) =N_0$
we obtain, after reordering the regions $R_j$, that  the only
possibility is:
 $\C(N_1,\fa)\cap U \subset R_1$ and
$\C(N_2,\fa)\cap U \subset R_3$.

\begin{Claim} We can choose a solid torus $V \subset M$ with boundary
$\po V$ invariant by $\fa$, such that $V$ contains a
neighborhood of $N_0$, \, it also contains $\Lambda(N_0)$,  and
$\sing(\fa) \cap V$ consists of $N_0$ and the two other components
$N_1$, $N_2$ of center-type.
\end{Claim}

 Let us prove this claim: given a leaf $L_j\subset \C(N_j,\fa)$ of $\fa$ that intersects
$U $ we have that $L_j\setminus (L_j\cap U )$ is a torus minus a strip,  {\it i.e.}, a cylinder. Thus,
by Proposition~\ref{Proposition:trivialneighborhood},
 given an outer leaf $L$ such that
$L\cap \C(N_j,\fa)=\emptyset$ one has that  $L$  is
obtained as the union of  two cylinders  glued by their boundaries, so
 $L$ is a torus. hence the basins $\C(N_j,\fa)$ are solid
tori. Proposition~\ref{Lemma:solidtorus} then shows that $L$
bounds a region diffeomorphic to the solid torus, obtained by the
union of two solid cylinders.
This region is invariant and contains the singularities $N_0, N_1,
N_2$, proving the claim and completing the proof of
(\ref{Lemma:eliminationalldimensionone}).
\end{proof}

\begin{Lemma}[$N_0\cong S^1$, non-isolated centers, nonsingular intersection of boundaries]
\label{Lemma:eliminationalldimensiononenonsingular} Suppose that
$c(\fa) >  s(\fa)$,  the components of the singular set, $N_0\subset \Sad(\fa)$ and $N_1, N_2
\subset \Cent(\fa)$, are all of dimension one,  and
$N_0\subset\partial \mathcal C(N_1,\fa) \cap
\partial\mathcal C(N_2,\fa)\not\subset \sing(\fa)$. Then
 there is an invariant solid torus $V$ containing
$N_0, N_1, N_2$ and we can replace  $\fa\big|_V$ by a
foliation by concentric tori.
\end{Lemma}

\begin{proof}
We use the same notation above. Considering the  several
possibilities for the trace of $\C(N_1,\fa)$ and $\C(N_2,\fa)$ in
$U $ under the assumption that $\C(N_1,\fa)\cap \C(N_2,\fa) =N_0$
we obtain, after reordering the regions $R_j$, that  the only
possibilities are:

\smallskip
\noindent{\bf First case}. $\C(N_1,\fa)\cup \C(N_2,\fa) \cup
\Lambda(N_0)$ is a neighborhood of $N_0$.

\smallskip
This corresponds to the case where  four regions $R_j$ are
intersected by the basins and the region $\Omega\subset M$
obtained as the union of these basins and $\Lambda(N_0)$ is open
and closed in $M$. Thus $\Omega=M$ and $\sing(\fa) = N_0 \cup N_1
\cup N_2$, contradicting our hypothesis that $c(\fa) > s(\fa)$.
This case cannot occur.

\noindent{\bf Second case}. In this case $\C(N_1,\fa)\cap U
\subset R_1$ and $\C(N_2,\fa)\cap U \subset R_2 \cup R_4$ with
$\C(N_2,\fa) \cap R_2 \ne \emptyset$ and $\C(N_2,\fa) \cap R_4 \ne
\emptyset$.

\begin{Claim} We can choose an invariant  solid torus  $V \subset M$ with boundary
$\po V$,  invariant by $\fa$ and such that $V$ contains a
neighborhood of $N_0$ ; also, $V$  contains $\Lambda(N_0)$ and
$\sing(\fa) \cap V$ consists of $N_0$ and two other components
$N_1$, $N_2$ of center-type.
\end{Claim}

\begin{proof}[Proof of the claim]
Using the same notation as above we consider an outer leaf  $L$  of
$\fa$. Then $L\cap U $ has two connected components (those close
to $\C(N_2,\fa) \cap U $) so that by Proposition
~\ref{Proposition:trivialneighborhood}, $L$ is the union of two
strips (the result of deleting two parallel strips in the torus)
glued by their boundaries, thus  resulting in a torus. This leaf bounds a
region which is the union of two solid cylinders glued by their
common boundaries, resulting in a solid torus.
\end{proof}

This completes the proof of the lemma.
\end{proof}

\begin{figure}[ht]
\begin{center}
\includegraphics[scale=0.45]{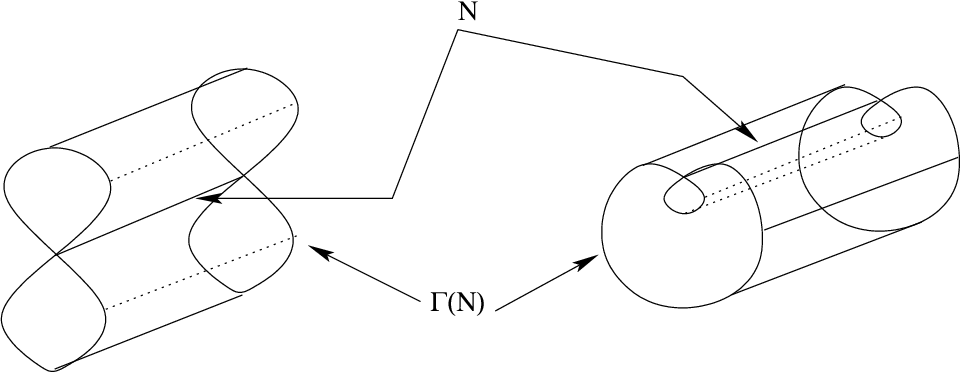}
\end{center}
\caption{} \label{Figure:twotori1}
\end{figure}

%%%\centerline{\bf Figure 1} \vglue.1in

\begin{figure}[ht]
\begin{center}
\includegraphics[scale=0.45]{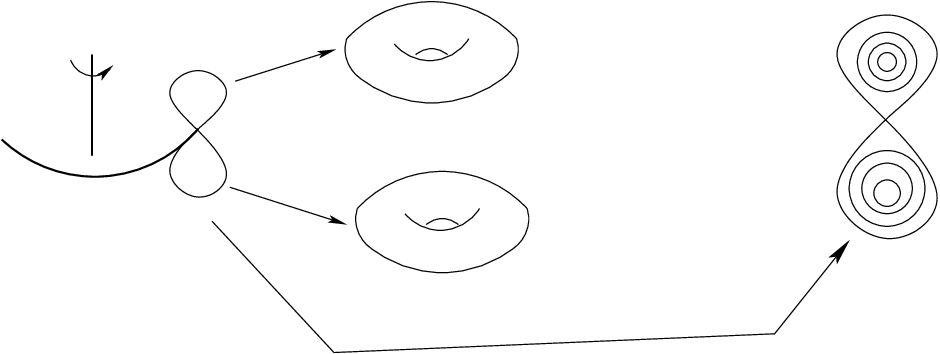}
\end{center}
\caption{} \label{Figure:twotori2}
\end{figure}

%%%\centerline{\bf Figure 1} \vglue.1in

\begin{figure}[ht]
\begin{center}
\includegraphics[scale=0.4]{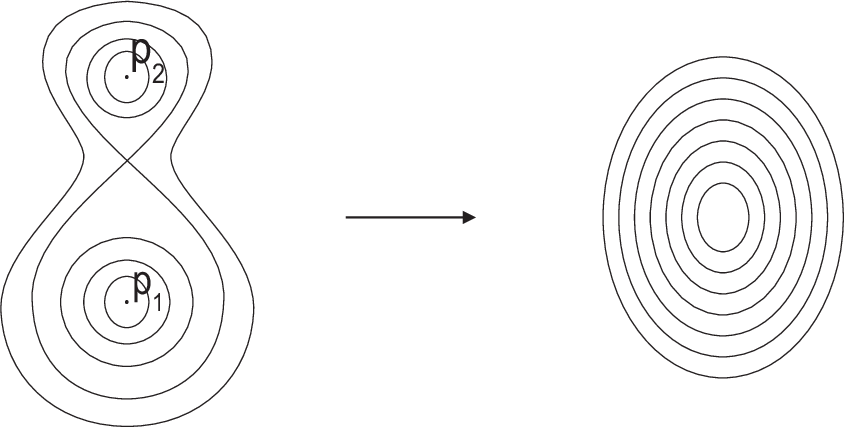}
\end{center}
\caption{} \label{Figure:eliminationoftwocentersbyone}
\end{figure}

\vglue.1in

\begin{figure}[ht]
\begin{center}
\includegraphics[scale=0.45]{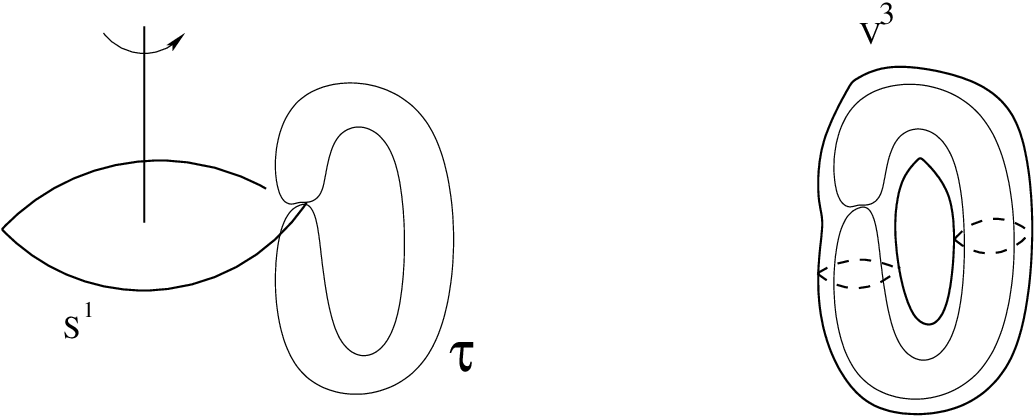}
\end{center}
\caption{}
\end{figure}

%%%\centerline{\bf Figure 1} \vglue.1in

\begin{figure}[ht]
\begin{center}
\includegraphics[scale=0.4]{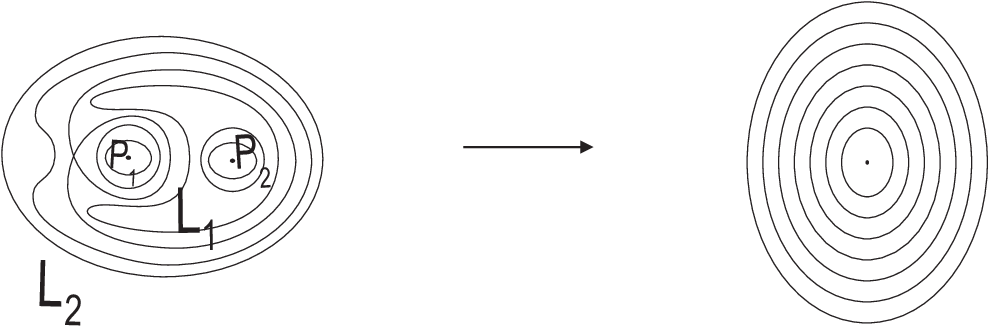}
\end{center}
\caption{Transverse section showing the elimination procedure}
\end{figure}

%%%\centerline{\bf Figure 1} \vglue.1in

The remaining cases  are the  content of the following lemmas:

\begin{Lemma}[$N_0\cong S^1$, mixed dimensions, singular intersection of boundaries]
\label{Lemma:nonisolatedmixedsingular} Assume now that
$N_0\subset \sing(\fa)$ is a non-isolated saddle component such
that $N_0=
\partial \mathcal C(N_1,\fa)\cap \partial \mathcal C(N_2,\fa)$ for
center-type components $N_1, N_2 \subset \sing(\fa)$, where $N_1$
is non-isolated and $N_2$ is isolated. Then there is a closed
invariant ball $B^3\subset M$ containing $N_0\cup N_1 \cup N_2$,
such that we can replace $\fa$ in $B^3$ by a foliation with an
isolated center.
\end{Lemma}

\begin{proof} We keep the notation of the proof of the above lemmas.
We have the following possibilities, up to reordering the regions
$R_1,...,R_4$ defined by $\Lambda(N_0)$ in the distinguished
neighborhood $U $ of the saddle $N_0$.

\vglue.1in \noindent{\bf First case}. $\C(N_2,\fa)\cap U \subset
R_1$ and $\partial C(N_1,\fa) \cap U  \subset R_3$. A leaf
$L_2\subset \C(N_2,\fa)$  is a $2$-sphere and $L_2\cap U $ is a
neighborhood of an equator in $L_2$, so that $L_2\setminus
(L_2\cap U )$ is the disjoint union of two discs $D^+$ and $D^-$.
A leaf $L_1\subset \C(N_1,\fa)$ is a torus and $L_1 \cap U $ is a
strip in the torus, so that $L_1\setminus (L_1\cap U )$ is a
cylinder. By Proposition~\ref{Proposition:trivialneighborhood},  an
exterior leaf $L$  to $\Lambda(N_0)$  is homeomorphic to the union
of the discs $D^+$ and $D^-$ and a cylinder, through their common
boundaries; so  $L$ is homeomorphic to the $2$-sphere. Also, $L$
bounds a region $\Omega$ containing the
 union $\C(N_1,\fa) \cup \C(N_2,\fa) \cup U $. This region is  homeomorphic to
 the union of a solid cylinder (obtained from the solid torus $\C(N_1,\fa)$ by deleting the
 neighborhood $U $ which is of product type by $S^1$) with  two solid hemispheres $H^+$ and
 $H^-$ (obtained from the solid ball $\C(N_2,\fa)$ by deleting the solid intersection
 with $U $) so that $\Omega$ is a solid ball.
The situation is depicted in  (Figure~\ref{Figure:sphere3}).

 \vglue.1in

 \noindent{\bf Second case}. $\C(N_1,\fa)$ and $\C(N_2,\fa)$
 intersect some adjacent regions.
 In this case the intersection $\partial C(N_1,\fa)\cap\partial
 \C(N_2,\fa)$ contains some separatrix of $N_0$, what is not
 possible by hypothesis.

 This proves the lemma.
 \end{proof}

\begin{Lemma}[$N_0\cong S^1$, mixed dimensions, non-singular intersection of boundaries]
\label{Lemma:nonisolatedmixednonsingular} Assume now that
$N_0\subset \sing(\fa)$ is a non-isolated saddle component such
that $N_0=
\partial \mathcal C(N_1,\fa)\cap \partial \mathcal C(N_2,\fa)\not \subset \sing(\fa)$ for
center-type components $N_1, N_2 \subset \sing(\fa)$, where $N_1$
is non-isolated and $N_2$ is isolated. Then there is a closed
invariant region $\Omega\subset M$ diffeomorphic to the product
$S^2 \times [0,1]$ or to the product $S^1 \times S^1 \times [0,1]$,
containing $N_0\cup N_1$, such that we can replace $\fa$ in
$\Omega$ by a regular foliation by $2$-spheres or tori
respectively. The center $N_2$ appears at infinity with respect to
$\Omega$.
\end{Lemma}

\begin{proof} We keep the notation of the proof of the above lemmas.
Since by hypothesis we have $N_0\subset \partial \C(N_1,\fa)\cap
\partial \C(N_2,\fa) \not\subset \sing(\fa)$, we have the following possibilities,
up to reordering the regions $R_1,...,R_4$ defined by
$\Lambda(N_0)$ in the distinguished neighborhood $U $.

\vglue.1in \noindent{\bf First case}.
 $\C(N_2,\fa)\cap U
\subset (R_2 \cup R_4)$, $\C(N_2,\fa)$ intersects $R_2$ and $R_4$,
$\partial C(N_1,\fa) \cap U  \subset R_3$ and $\C(N_1,\fa)\cap
R_1=\emptyset$. As in the preceding lemmas, a leaf
$L_2\subset \C(N_2,\fa)$ is a $2$-sphere and therefore its trace
in $U $ consists of two parallel strips,  {\it i.e.}, two
neighborhoods of meridians,  and $L_2\setminus (L_2\cap U )$ is
the union of two disjoint discs. A leaf $L_1\subset \C(N_1,\fa)$
is a torus and its trace in $U $ is a strip, so that $L_1\setminus
(L_1\cap U )$ is a cylinder. A leaf $L$ intersecting $R_1$ cannot
belong to $\C(N_1,\fa)$ nor to $\C(N_2,\fa)$. Nevertheless the
above remarks and
Proposition~\ref{Proposition:trivialneighborhood} imply that $L$
is also a $2$-sphere. Thus we can find a region $\Omega$ bounded
by leaves in $\C(N_2,\fa)$ diffeomorphic to $S^2$, and by leaves
intersecting  $R_1$, also diffeomorphic to $S^2$. The region
$\Omega$ contains $N_0$ and $N_1$ and is diffeomorphic to the
product $S^2 \times [0,1]$.

\vglue.1in

\noindent{\bf Second case}.  $\C(N_1,\fa)\cap U  \subset (R_2 \cup
R_4)$, $\C(N_1,\fa)$ intersects $R_2$ and $R_4$, $\partial
C(N_2,\fa) \cap U  \subset R_3$ and $\C(N_2,\fa)\cap
R_1=\emptyset$. In this case we apply
Proposition~\ref{Proposition:trivialneighborhood} to compare the
leaves in $\C(N_1,\fa)$ and the leaves in $\C(N_2,\fa)$ also
through their traces in $U $. This implies that a torus minus two
neighborhoods of parallel meridians is homeomorphic to a
$2$-sphere minus a neighborhood of the equator, what is a
contradiction. Hence this case cannot happen.

\vglue.1in \noindent{\bf Third case}.  $\C(N_1,\fa)\cap U \subset
(R_2 \cup R_4)$, $\C(N_1,\fa)$ intersects $R_2$ and $R_4$,
$\partial C(N_2,\fa) \cap U  \subset R_1$ and $\C(N_2,\fa)\cap
R_3=\emptyset$. In this case we apply
Proposition~\ref{Proposition:trivialneighborhood} to
compare   the leaves in $\C(N_1,\fa)\setminus (\C(N_1,\fa)
\cap U )$ and the leaves in $\C(N_2,\fa)\setminus (\C(N_2,\fa)\cap
U )$  through their traces in $U $. We conclude that the leaves
intersecting $R_3$ are tori and  we obtain an
invariant  region  $\Omega\subset M$ diffeomorphic to $S^1 \times
S^1 \times [0,1]$, containing $N_1$ and $N_0$, where we can replace
the foliation $\fa$ by a non-singular foliation by tori. The
center $N_2$ is at infinity with respect to $\Omega$.

\end{proof}

From the above lemmas we have:

\begin{Proposition}
\label{Proposition:nonisolatedsaddle} Let $\fa$ be a closed
Bott-Morse foliation on a closed $3$-manifold $M$ such that
$c(\fa) > 2 s(\fa)$ or else $\sing(\fa)$ has pure dimension and
$c(\fa) > s(\fa)$. Suppose that there is a non-isolated saddle
$N_0\subset\sing(\fa)$ paired with two centers $N_1,
N_2\subset \sing(\fa)$. Then we have the following possibilities:

\begin{itemize}
\item[{\rm(i)}] If  $N_1$ and $N_2$ are isolated, then the pairing
$N_0\leftrightarrow (N_1,N_2)$ is trivial and we have the
following possibilities:
\begin{itemize}
\item[{\rm (i.1)}]  $\partial\mathcal C(N_1, \fa)\cap
\partial\mathcal C(N_2, \fa)\not\subset \sing(\fa)$ and there is
 a solid cylinder $V\simeq D^2\times [0,1]$ with boundary
$\po V$ consisting of two invariant discs $D_1\cong D^2\times
\{0\}$ and $D_2 \simeq D^2\times \{1\}$ and a transverse open
cylinder $\Sigma\cong S^1\times (0,1)$. We have $\sing(\fa) \cap
V=N_0\cup N_1\cup N_2$ and $\fa\big|_V$ can be replaced by a
trivial foliation by discs.

\item[{\rm (i.2)}] If $\partial\mathcal C(N_1, \fa)\cap
\partial\mathcal C(N_2, \fa)=N_0$, then there is an invariant
region $R$ diffeomorphic to $S^2 \times [0,1]$ and containing
$N_0, N_1,N_2$, so that we can replace $\fa$ in $R$ by a regular
foliation by $2$-spheres.
\end{itemize}

\item[{\rm(ii)}] If  $N_1$ and $N_2$ are non-isolated, then there
is an invariant solid torus $V$ containing $N_0, N_1, N_2$ and we
can replace  $\fa\big|_V$ by a foliation by concentric tori
having a single non-isolated center as singular set.

\item[{\rm(iii)}] If $N_1$ is non-isolated and $N_2$ is isolated,
then  there is a closed invariant ball $B^3\subset M$ containing
$N_0\cup N_1 \cup N_2$, such that we can replace $\fa$ in $B^3$ by
a foliation with an isolated center.

\end{itemize}
 \end{Proposition}

We compile the
information in Propositions~\ref{Proposition:isolatedsaddlecases}
and \ref{Proposition:nonisolatedsaddle} in the following theorem:

\begin{Theorem}
\label{Theorem:onlywhatweneed} Let $\fa$ be a closed Bott-Morse
foliation on a closed oriented three-manifold $M$ such that either
$c(\fa) > 2 s(\fa)$ or else $\sing(\fa)$ has pure dimension and
$c(\fa) > s(\fa)$. Suppose we have two different center components
$N_1, N_2 \in \Cent(\fa)$ and a saddle component $N_0 \in
\Sad(\fa)$ such that $N_0 \subset
\partial \mathcal C(N_1,\fa) \cap \partial \mathcal C(N_2,\fa)$.
We have the following possibilities:

\begin{itemize}

\item[{\rm(1)}] The saddle $N_0$ is isolated. Then either $N_0, N_1, N_2$
belong to a same disconnected irreducible component of $\fa$ or
one of the centers, say $N_1$, must be isolated and we have the following possibilities:

\item[{\rm(1.i)}] If $N_2$ is also isolated, then one of the
pairings $N_0\leftrightarrow N_j$ is trivial.

\item[{\rm(1.ii)}] If $N_2$ is non-isolated, then there are two
possibilities: either the pairing $N_0\leftrightarrow N_1$ is
trivial or $\fa$ can be modified in  a solid torus replacing
$N_0\cup N_1$ by  a one-dimensional center.

\item[{\rm ($\star$)}] If $N_0, N_1, N_2$ belong to a same
irreducible component, then there is an invariant compact region
$\Omega\subset M$ containing these singularities and the
restriction $\tilde \fa$ of $\fa$ to $\tilde M=M \setminus \Omega$
is a closed Bott-Morse foliation which satisfies $c(\tilde \fa) > 2
s(\tilde\fa) \geq 2$. Also, given a center
$N\subset\sing(\fa)\setminus (N_1 \cup N_2)$ we have $ \C(N,\fa)
\cap \Omega=\emptyset$.

\item[{\rm(2)}]  The saddle $N_0$ is non-isolated. Then we have the following possibilities:

\item[{\rm(2.i)}] If  $N_1$ and $N_2$ are isolated, then the
pairing $N_0\leftrightarrow (N_1,N_2)$ is trivial.

\item[{\rm(2.ii)}] If  $N_1$ and $N_2$ are non-isolated, then we
can modify $\fa$ in an invariant solid torus replacing  $N_0, N_1,
N_2$ by  a single non-isolated center.

\item[{\rm(2.iii)}] If $N_1$ is non-isolated and $N_2$ is
isolated, then we can modify $\fa$ in an invariant solid ball
replacing $N_0\cup N_1 \cup N_2$ by  an isolated center.

\end{itemize}

\end{Theorem}

\begin{proof}

Part (2) follows from
Proposition~\ref{Proposition:nonisolatedsaddle}. Parts (1.i),
(1.ii) follow from
Proposition~\ref{Proposition:isolatedsaddlecases} while  ($\star$)
follows from that same proposition,
Remark~\ref{Remark:irreduciblecase} and the following
argumentation:
 Suppose  that $N_0, N_1, N_2$ are in a same
irreducible component $\Omega \subset M$. Since $c(\fa) > 2s(\fa)$
there is another center $N_3$ distinct from $N_1, N_2$ and such
that $\partial \C(N_3, \fa)\cap N_0 = \emptyset$. If $\partial
\C(N_3,\fa)=\emptyset$ then by Theorem~\ref{Lemma:basinopen} (ii) we
have $s(\fa)=0$, absurd. Thus by Theorem~\ref{Lemma:basinopen} (iii)
we have $\partial \C(N_3,\fa)=N_0 ^\prime$ for a saddle
$N_0^\prime \ne N_0$. This shows that $s(\fa) \geq 2$. Thus we can
consider the foliation $\fa$ restricted to the manifold with
boundary $\tilde M:= M \setminus \Omega$. This is a Bott-Morse
foliation $\tilde \fa$ with $\sing(\tilde\fa) = \sing(\fa)
\setminus N_1 \cup N_2 \cup N_0$, so that it satisfies the
inequality $c(\tilde \fa) > 2s(\tilde \fa)$, and   $s(\tilde \fa)
\geq 1$.
\end{proof}

\section{A classification theorem in dimension three}

We now prove:

\begin{Theorem}
\label{Theorem:Center-Saddledimensionthree}
\label{Theorem:Center-Saddledimensionthreecircles}   Let $M^3$ be
a closed oriented connected $3$-manifold equipped with a closed
 Bott-Morse foliation $\fa$ satisfying either $c(\fa) > 2 s(\fa)$ or else $c(\fa) >  s(\fa)$ and its singular set $sing(\fa)$ is pure-dimensional.
 Then{\rm:}

{\rm(i)} there is a deformation of $\fa$ via foliated surgery into
a compact Bott-Morse
 foliation on $M$, and  the holonomy pseudogroup of the foliation is
 preserved off the singular
set; and

 {\rm(ii)}   $M$ is homeomorphic to the 3-sphere, a Lens space or
$S^1\times S^2$.
\end{Theorem}

 Part (ii) of
Theorem~\ref{Theorem:Center-Saddledimensionthree} is an immediate
consequence of part (i) and Theorem D in \cite{Se-Sc}. Examples in
Section~\ref{Section:examples} show that the hypothesis that
either $c(\fa) > 2 s(\fa)$ or $c(\fa) > s(\fa)$ and the singular
set has pure dimension,  is necessary. If $\sing(\fa)$ has
dimension $0$ then (i)  is proved in \cite{Camacho-Scardua}; in
this case $M$ is actually the 3-sphere. \vglue.1in To prove this
theorem we need the following elimination result:

\begin{Proposition}
\label{Proposition:eliminationBottMorse} Let $\fa$ be a closed
Bott-Morse foliation on a connected closed $3$-manifold $M$.
Suppose that $c(\fa) > 2 s(\fa)$. Then either $s(\fa)=0$ or $\fa$
admits a modification  into a closed Bott-Morse foliation $\fa_1$
on $M$ such that either $c(\fa_1)=c(\fa)-1$ and $s(\fa_1) = s(\fa)
-1$ or $c(\fa_1)=c(\fa)-2$ and $s(\fa_1) = s(\fa)-1$.
\end{Proposition}

\begin{proof}
 We fix a component $N\in
\Cent(\fa)$ and consider $\mathcal C(N,\fa)$ as usual. If $\po
\mathcal C(N,\fa)=\emptyset$ then we apply
Theorem~\ref{Lemma:basinopen} (ii) to conclude that
$\Sad(\fa)=\emptyset$ and therefore $s(\fa)=0$. Assume therefore
that $\po \mathcal C(N,\fa)\ne \emptyset$. Then by
Theorem~\ref{Lemma:basinopen} (iii) we have $\po\mathcal
C(N,\fa)\subset \Sad(\fa)$ and there is some component $S(N)\in
\Sad(\fa) \cap \partial \mathcal C(N,\fa)$. In this way we can define a
map $\Theta$   from the set of center-type
components $\Cent(\fa)$ into the set of saddle-type
 components $\Sad(\fa)$, given by $  N \mapsto S(N)$. This map cannot be injective
since we have the inequality $c(\fa) > 2 s(\fa)$. Therefore there
are at least two different center components $N_1, N_2 $ and
a saddle component $N_0$ such that $N_0 \subset
\partial \mathcal C(N_1,\fa) \cap \partial \mathcal C(N_2,\fa)$.

If $N_0$ is non-isolated the conclusion  follows from
Theorem~\ref{Theorem:onlywhatweneed} part (2).  Assume now that
$N_0$ is an isolated saddle point. According to
Theorem~\ref{Theorem:onlywhatweneed} part (1)  the reduction
exists if $N_0, N_1, N_2$  are not in a same irreducible
component of $\fa$. Suppose therefore that $N_0, N_1, N_2$ are in
a same irreducible component $\Omega \subset M$. Then
Theorem~\ref{Theorem:onlywhatweneed} ($\star$) says that there is an invariant
compact region $\Omega\subset M$ containing these singularities
such that:
\begin{itemize}
\item The restriction $\tilde \fa$ of $\fa$ to $\tilde M=M \setminus
\Omega$ is a closed Bott-Morse foliation which satisfies $c(\tilde \fa) > 2
s(\tilde\fa) \geq 2$; and
\item  $\C(N,\fa)\cap \Omega=
\emptyset$ for every other center $N\subset \sing(\fa)$ one has $N\ne N_j$, $
j=1,2$.
\end{itemize}
Thus, arguing as above we can find two centers $\tilde
N_1, \tilde N_2$ and a saddle $\tilde N_0$ such that $\partial
\C(N_1, \tilde \fa) \cap \tilde \C(N_2,\tilde \fa) = \tilde N_0$.
If $\tilde N_0$ is not isolated or $\tilde N_0, \tilde N_1, \tilde
N_2$ do not belong to the same irreducible component of $\tilde
\fa$ on $\tilde M$, then  we can obtain a closed Bott-Morse
foliation  $\tilde{\tilde \fa}$ which is a reduction of $\tilde
\fa$ on $\tilde M$ and  satisfies $c(\tilde{\tilde \fa})> 2
s(\tilde{\tilde \fa})$. This gives a reduction of $\fa$ on $M$
with the desired properties.
The remaining case is when
$\tilde N_0, \tilde N_1, \tilde N_2$ belong to a same irreducible
component $\tilde \Omega$ of $\tilde \fa$ on $\tilde M$. Notice that for
any irreducible component we have two centers and one saddle.

This process can be repeated until we can assure that  for
the resulting foliation, the two centers and the saddle are not in
a same irreducible component, for otherwise, if
all singularities are in irreducible components, then we would have
$c(\fa)= 2s(\fa)$,  which is a contradiction.
 \end{proof}

\begin{proof} [Proof of
Theorem~\ref{Theorem:Center-Saddledimensionthree}]   Assume that  $c(\fa) > 2 s(\fa)$. We proceed by
induction on the number $s(\fa)$ of saddle components in
$\sing(\fa)$. We first assume that $s(\fa)=0$. Then $\fa$ is a
compact Bott-Morse foliation with nonempty singular set and
Theorem~\ref{Theorem:Center-Saddledimensionthree} follows from
Theorem \ref{Theorem D}. Now we assume that the result is valid
for closed Bott-Morse foliations $\fa_1$ on closed $3$-manifolds
satisfying $c(\fa_1) > 2\, s(\fa_1)$ and having a number of saddle
components not greater than $k\ge 1$, {\it i.e.}, $s(\fa_1) \leq
k$. Let $\fa$ be a closed Bott-Morse foliation on $M$ such that
$c(\fa) > 2 s(\fa)$ and such that $s(\fa)=k+1$. Since $c(\fa) > 2
s(\fa)$ and $s(\fa) \geq 1$ it follows from
Proposition~\ref{Proposition:eliminationBottMorse} that $M$ admits
a closed Bott-Morse foliation $\fa_1$ for which we have $c(\fa_1)>
2\, s(\fa_1)$ and $s(\fa_1) < s(\fa)$. This proves the theorem.
Now we assume that $c(\fa) >  s(\fa)$ and the singular set of $\fa$
is pure-dimensional. The result is proved in \cite{Camacho-Scardua} in case all components
of the singular set have dimension zero. Thus we only have to
prove the case where  the components of the singular set have
dimension one. For this it is enough to observe that if we have a
pairing $N_0 \leftrightarrow (N_1,N_2)$ where all the components
are of dimension one then we can always eliminate some pair of
components $N_0 \leftrightarrow N_j$ and remain with a center
component. Therefore, in case $\sing(\fa)$ is a union of dimension
one components, it is  enough to have $c(\fa) > s(\fa)$ to get the
same conclusions as above.

\end{proof}

\bibliographystyle{amsalpha}

\begin{thebibliography}{31}
\frenchspacing



\bibitem{Bott1} R. Bott: {\it Nondegenerate critical manifolds};
Annals of Math., vol. 60, \# 2 (1954), p. 248-261.


\bibitem{Camacho-Lins Neto} C. Camacho, A. Lins Neto:
{\it  Geometric theory of foliations}. Translated from the Portuguese by
Sue E. Goodman. Birkh\"auser Boston, Inc., Boston, MA, 1985.



\bibitem{Camacho-Scardua} C. Camacho, B.
Sc\'ardua: {\it On codimension one foliations with Morse
singularities on three-manifolds}; Topology and its Applications,
vol. 154,  \#  6 (2007), p. 1032-1040.


\bibitem{Camacho-Scardua2} C. Camacho, B.
Sc\'ardua: {\it Foliations with Morse singularities}; Proc. Am. Math. Soc. 136, \# 11  (2008), p. 4065-4073.

\bibitem{Ee-Ku1}   J. Eells, N. Kuiper: {\it
Manifolds which are like projective planes}; Pub. I.H.E.S., vol. 14
(1962), p. 5-46.

\bibitem{Ee-Ku2} J. Eells, N. Kuiper: {\it Closed manifolds
which admit nondegenerate functions with three critical points};
Indag. Math. 23 (1961), p. 411-417.

\bibitem{Er-Sh} C. Ehresmann, W. Shih: Sur les espaces feuillet\'es: th\'eor\`eme de
stabilit\'e, C.R. Acad. Sci. Paris 243 (1956), p. 344-346.



\bibitem{Go} C. Godbillon:  {\it Feuilletages, \'Etudes geom\'etriques};
Progress in Mathematics, 98, Birkh\"auser Verlag (1991).


\bibitem{Hempel} J. Hempel: {\it $3$-Manifolds}, AMS Chelsea Publishing,
Providence USA (1976).



\bibitem{Hur-Tob} S. Hurder, D. T\"oben:  {\it Transverse LS-Category
    for Riemannian foliations}; Trans. Am. Math. Soc. 361, \# 11 (2009), p. 5647-5680.

\bibitem{LSV} D. T. L\^e, J. Seade,
A.  Verjovsky: {\it Quadrics, orthogonal actions and involutions
in complex projective spaces}; Ens. Math.  49 (2003), \# 1-2, p.
173-203.

\bibitem{Milnor1} J. Milnor,  {\it
A procedure for killing homotopy groups of differentiable
manifolds}. Proc. Sympos. Pure Math., Vol. III   (1961), p. 39-55, A. M. S.


\bibitem{Milnor2} J. Milnor: {\it Morse Theory};
Annals of Mathematics Studies, No. 51, Princeton University Press,
Princeton, N.J., 1963.


\bibitem{Milnor3} J. Milnor: {\it Singular points of complex hypersurfaces};
Annals of Mathematics Studies, No. 61, Princeton University Press,
Princeton, N.J., 1968.


\bibitem{Mo}  P. Molino: {\it Riemannian Foliations}; Progress in
  Mathematics vol. 73 (1988), Birkh\"auser.


\bibitem{Morse1} M. Morse:
{\it The calculus of variations in the large}; American Math. Soc.
Colloquium Publications, 18. AMS, Providence, RI, 1996.

\bibitem{Reebthesis}  G. Reeb: {\it  Sur
certaines propri\'et\'es topologiques des vari\'et\'es
feuillet\'ees};  Publ. Inst. Math. Univ. Strasbourg 11, pp. 5--89,
155--156. Actualit\'es Sci. Ind., no. 1183 Hermann \& Cie., Paris,
1952.




\bibitem{Reeb2} G. Reeb: {\it Sur les
points singuliers d'une forme de Pfaff compl\'etement int\'egrable
on d'une fonction num\'erique}; C.R.A.S. Paris 222 (1946), p.
847-849.




\bibitem{Se-Sc} B. Scardua, J. Seade: {\it Codimension one foliations
with Bott-Morse singularities I}; J. Differ. Geom. 83, No. 1 (2009), p. 189-212.

\bibitem{Sus} H. J. Sussmann: {\it Orbits of families of vector fields and
    integrability of distributions}; Trans. A.M.S. 180(1973), p. 171-188.


\bibitem{Thom} R. Thom: {\it G\'en\'eralization de la th\'eorie de Morse aux
variet\'es feuillet\'ees};  Ann. Inst. Fourier (Grenoble) 14,
fasc. 1 (1964), p. 173-189.




\end{thebibliography}

\vglue.2in
\begin{tabular}{ll}
Bruno Sc\'ardua  & \qquad  Jos\'e Seade\\
Instituto de  Matem\'atica & \qquad Instituto de Matem\'aticas,   Unidad Cuernavaca  \\
Universidade Federal do Rio de Janeiro  & \qquad Universidad Nacional Aut\'onoma de M\'exico,\\
Caixa Postal 68530 & \qquad Av. Universidad s/n, Col. Lomas de Chamilpa\\
21.945-970, Rio de Janeiro-RJ   & \qquad   C.P. 62210, Cuernavaca, Morelos\\
BRAZIL &  \qquad M\'EXICO  \\
 scardua@impa.br & \qquad jseade@matcuer.unam.mx
\end{tabular}

\end{document}